\documentclass[11pt]{article}
\usepackage{graphicx}
 \usepackage{mathptmx,color}      
\usepackage[titletoc,title]{appendix}
\usepackage[numbers]{natbib}
\usepackage[T1]{fontenc}
\usepackage[applemac]{inputenc}
\usepackage[all]{xy}
\usepackage{enumerate}
\usepackage{comment}
\usepackage{amsmath,amssymb}
\usepackage[british]{babel}
\usepackage{a4wide}
\usepackage[sans]{dsfont}
\usepackage{amsfonts}
\usepackage{amsthm}
\usepackage[mathscr]{euscript}
\usepackage{graphicx}
\usepackage{mathrsfs}
\usepackage{latexsym}
\usepackage{setspace}
\usepackage{amsthm,amsmath}

\newcommand{\p}{^}

\newcommand{\pb}[2]{
	\ensuremath{\langle #1,#2 \rangle}}

\renewcommand{\d}[1]{{\rm d}#1}
	
\newcommand{\wt}{\widetilde}

	\newcommand{\rmd}{\mathrm{d}}
\newcommand{\rme}{\mathrm{e}}

\newcommand{\R}{\mathrm{R}}
\newcommand{\Lrm}{\mathrm{L}}

\newcommand{\Ascr}{\mathscr{A}}
\newcommand{\Bscr}{\mathscr{B}}
\newcommand{\Cscr}{\mathscr{C}}

\newcommand{\Escr}{\mathscr{E}}
\newcommand{\Fscr}{\mathscr{F}}
\newcommand{\Gscr}{\mathscr{G}}
\newcommand{\Hscr}{\mathscr{H}}

\newcommand{\Kscr}{\mathscr{K}}

\newcommand{\Nscr}{\mathscr{N}}

\newcommand{\Pscr}{\mathscr{P}}

\newcommand{\Rscr}{\mathscr{R}}
\newcommand{\Sscr}{\mathscr{S}}

\newcommand{\zt}{\zeta}

\newcommand{\al}{\alpha}

\newcommand{\ep}{\varepsilon}
\newcommand{\lm}{\lambda}

\newcommand{\Lm}{\Lambda}
\newcommand{\sig}{\sigma}

\newcommand{\om}{\omega}
\newcommand{\Om}{\Omega}
\newcommand{\cadlag}{c\`adl\`ag}

\newcommand{\Abb}{\mathbb{A}}
\newcommand{\Bb}{\mathbb{B}}
\newcommand{\Ebb}{\mathbb{E}}

\newcommand{\Fbb}{\mathbb{F}}

\newcommand{\Pbb}{\mathbb{P}}
\newcommand{\Qbb}{\mathbb{Q}}
\newcommand{\Rbb}{\mathbb{R}}

\newcommand{\Hbb}{\mathbb{H}}
\newcommand{\Gbb}{\mathbb{G}}

\newcommand{\bi}{\begin{itemize}}
\newcommand{\ei}{\end{itemize}}

\newcommand{\be}{\begin{enumerate}}
\newcommand{\ee}{\end{enumerate}}
\newcommand{\beq}{\begin{equation}}
\newcommand{\eeq}{\end{equation}}
\newcommand{\beqs}{\begin{equation*}}
\newcommand{\eeqs}{\end{equation*}}
\newcommand{\beqa}{\begin{eqnarray}}
\newcommand{\eeqa}{\end{eqnarray}}
\newcommand{\beqas}{\begin{eqnarray*}}
\newcommand{\eeqas}{\end{eqnarray*}}

\newcommand{\aPP}[2]{\ensuremath{\langle #1,#2 \rangle}}

\newtheorem{theorem}{Theorem}[section]
\newtheorem{lemma}[theorem]{Lemma}
\newtheorem{proposition}[theorem]{Proposition}
\newtheorem{corollary}[theorem]{Corollary}
\theoremstyle{definition}
\newtheorem{definition}[theorem]{Definition}

\newtheorem{assumption}[theorem]{Assumptions}

\newtheorem{remark}[theorem]{Remark}

\numberwithin{equation}{section}

\makeatletter
\def\timenow{\@tempcnta\time
\@tempcntb\@tempcnta
\divide\@tempcntb60
\ifnum10>\@tempcntb0\fi\number\@tempcntb
:\multiply\@tempcntb60
\advance\@tempcnta-\@tempcntb
\ifnum10>\@tempcnta0\fi\number\@tempcnta}
\makeatother

\title{On~The~Weak~Representation~Property~in Progressively~Enlarged~Filtrations\\ with an Application in\\ Exponential~Utility~Maximization}
\author{Paolo Di Tella$^1$}
\date{}
\newcommand{\pa}{\color{black}}
\begin{document}
\maketitle

\begin{abstract} In this paper we show that the weak representation property of a semimartingale $X$ with respect to a filtration $\mathbb{F}$ is preserved in the progressive enlargement $\mathbb{G}$ by a random time $\tau$ avoiding $\mathbb{F}$-stopping times and such that $\Fbb$ is immersed in $\mathbb{G}$. As an application of this, we can solve an exponential utility maximization problem in the enlarged filtration $\mathbb{G}$ following the dynamical approach, based on suitable BSDEs, both over the fixed-time horizon $[0,T]$, $T>0$, and over the random-time horizon $[0,T\wedge\tau]$. 
 \end{abstract}
\vspace{1em} 

{\noindent
\footnotetext[1]{ \emph{Adress:} Inst.\ f\"ur Mathematische Stochastik, TU Dresden. Zellescher Weg 12-14 01069 Dresden, Germany. 

\emph{E-Mail: }{\tt Paolo.Di\_Tella{\rm@}tu-dresden.de}}}

\section{Introduction}\label{sec:intro}
In this paper we consider an $\Rbb\p d$-valued semimartingale $X$ that possesses the weak representation property (from now on WRP) with respect to a filtration $\Fbb$. For the definition of the WRP see Definition \ref{eq:wprp} below. We denote by $\Gbb$ the progressive enlargement of $\Fbb$ by a random time $\tau$ and assume that $\tau$ avoids $\Fbb$-stopping times and that $\Fbb$ is immersed in $\Gbb$ (Assumptions \ref{ass:av.im} below). We then show that under these assumptions on $\tau$ the WRP of $X$ with respect to $\Fbb$ propagates to $\Gbb$. More precisely, if $H$ denotes the default process associated with $\tau$, we prove that the $\Rbb\p{d+1}$-valued semimartingale $(X,H)$ possesses the WRP with respect to $\Gbb$. The WRP which we obtain with respect to $\Gbb$ is valid for all $\Gbb$-martingales and not only for $\Gbb$-martingales stopped at $\tau$.

Let $X$ be an $\Fbb$-local martingale and let $\Gbb$ be the progressive enlargement of $\Fbb$ by a random time $\tau$. The propagation of the predictable representation property (from now on PRP) of $X$ to $\Gbb$ has been extensively studied in the literature under different assumptions on $\tau$ and a pioneering work about this topic is Kusuoka \cite{Ku99} in a Brownian setting.  In \cite{AJR18}, Aksamit et al.\ proved a PRP for a class of local martingales stopped at $\tau$ if $\Gbb$ is the progressive enlargement of the filtration generated by a Poisson process. In Jeanblanc \& Song \cite{JS15} the propagation of the PRP to $\Gbb$ is studied in a more general case. We also recall Coculescu, Jeanblanc \& Nikeghbali \cite{CJN12} for results about the propagation of the PRP.

The result on the WRP with respect to $\Gbb$ in this paper, generalises, in particular, the martingale representation theorem obtained in \cite{Ku99}. It also extends the WRP obtained by Kunita \& Watanabe in \cite{KW67} (see also \cite{K04}) for L\'evy processes to progressively enlarged L\'evy filtrations. Furthermore, it adds a class of non trivial examples to Becherer \cite{Be06}.

We remark that, according to \cite[Theorem 13.14]{HWY92}, if $X$ possesses the PRP then it also possess the WRP (for this reason the PRP is sometimes also called strong representation property). However, the converse is, in general, not true and it is easy to find martingales possessing the WRP but not the PRP.  This means that the study of the propagation of the WRP to the progressive enlargement $\Gbb$ is of its own interest. 

{\pa As an application of the WRP with respect to the progressive enlargement} $\Gbb$ of $\Fbb$, we can solve a problem of expected exponential utility maximization of the wealth of an investor with respect to $\Gbb$, following the dynamical approach, based on BSDEs. Indeed, because of the WRP with respect to $\Gbb$ obtained in this paper, {\pa this problem} fits in the frame of Becherer \cite{Be06} and we can apply \cite[Theorem 3.5]{Be06} to ensure the existence and the uniqueness of the solution of the involved BSDEs. 

In the financial context of the exponential utility maximization, the filtration $\Fbb$ models the information available in the market. On the other side, the random time $\tau$, which is not an $\Fbb$-stopping time, models the occurrence time of an external event, as the default of (part of) the market or the death of the investor himself, which cannot be inferred on the basis of the information available in $\Fbb$.
We consider the  exponential utility maximization problem with respect to $\Gbb$ both at maturity $T>0$ and, inspired by Jeanblanc et al.\ \cite{JMPR15}, at $T\wedge\tau$. The problem at $T$ describes the situation in which the investor (or his heirs) is able to go on with the investment up to the maturity $T$, although the exogenous event has occurred before $T$. Contrarily, the optimization problem at $T\wedge\tau$ corresponds to the case in which the investment can only be pursued up to the occurrence time $\tau$ of the market-exogenous shock event, that is, the investor has access to the market only up to $\tau$. This latter situation is common, e.g., in life insurance, where the time to maturity $T$ can exceed with probability one the life duration of the investor (in this case $\tau$ models the death of the investor). 

A few references about utility maximization in progressively enlarged filtrations are, e.g.,  Bielecki et al.\ \cite{BJ08}, \cite{BJR04}, Jeanblanc et al.\ \cite{JMPR15}, Lim \& Quenez \cite{LQ11}, for the dynamical approach and Blanchet-Scalliet et al.\ \cite{BSEKJM08} and Bouchard \& Pham \cite{BoPh04} for the duality approach. 

Concerning martingale representation theorems and utility maximization in progressively enlarged filtration, we also mention some work-in-progress as Choulli et al. \cite{CDV18} and \cite{CY18}.

This paper has the following structure: In Section \ref{sec:bas}, we introduce some preliminary notions. In Section \ref{sec:WRP}, we recall properties of integer-valued random measures and discuss the WRP for a general semimartingale. Section \ref{sec:prog.en} is a short summary of results about progressive enlargement of filtrations which will be useful later. In Section \ref{sec:wea.prp.G}, we show that the WRP of a semimartingale $X$ propagates from $\Fbb$ to its progressive enlargement $\Gbb$. Finally, in Section \ref{sec:app.ex.ut.max}, as a consequence of the WRP with respect to $\Gbb$ obtained in Section \ref{sec:wea.prp.G}  and the theory developed in \cite{Be06}, we solve an exponential utility maximization problem both over $[0,T]$ and over the random-time horizon $[0,T\wedge\tau]$. 

\section{Basic Notions}\label{sec:bas}
 In this paper we regard an element of $\Rbb\p d$, $d\geq1$, as a column, that is,  $v=(v\p1,\ldots,v\p d)\p{tr}$, $v\in\Rbb\p d$, where $tr$ denotes the transposition operation.

Let $(\Om,\Fscr,\Pbb)$ be a complete probability space and let $\Fbb=(\Fscr_t)_{t\geq0}$ be a filtration satisfying the usual conditions, i.e., $\Fbb$ is right-continuous and contains the $\Pbb$-null sets of $\Fscr$. For a filtration $\Fbb$ we define $\Fscr_\infty$ as the smallest $\sig$-algebra containing each $\Fscr_t$, $t\geq0$, that is $\Fscr_\infty:=\bigvee_{t\geq0}\Fscr_t$.

For a \cadlag\  process $X$, we define $X_{0-}:=X_0$ so that the jump process $\Delta X$ is equal to zero in $t=0$.

If $X$ is a \cadlag\ $\Fbb$-adapted process with paths of finite variation and $K$ is a nonnegative measurable process, we denote by $K\cdot X=(K\cdot X_t)_{t\geq0}$ the (Stieltjes--Lebesgue) integral process of $K$ with respect to $X$, that is, $K\cdot X_t(\om):=\int_0\p t K_s(\om)\rmd X_s(\om)$.

As in \cite[I.3.6]{JS00}, we denote by $\Ascr\p+=\Ascr\p+(\Fbb)$ the set of \emph{integrable} increasing processes, that is, if $X\in\Ascr\p+$, then $X$ is $\Fbb$-adapted, has \cadlag\ and increasing paths, $X_0=0$ and $\Ebb[X_\infty]<+\infty$. By $\Ascr\p+_\mathrm{loc}=\Ascr\p+_\mathrm{loc}(\Fbb)$ we denote the space of $\Fbb$-adapted and \emph{locally integrable} increasing processes.
 
For $p\in[1,+\infty)$, we denote by $\Hscr\p p=\Hscr\p p(\Fbb)$ the space of $\Fbb$-adapted real-valued uniformly integrable martingales such that $\|X\|_{\Hscr\p p}:=\Ebb[\sup_{t\geq0}|X_t|\p p]\p{1/p}<+\infty$. By Doob's inequality, $\|X\|_{\Hscr\p p}$ is equivalent to $\|X\|_{p}:=\Ebb[|X_\infty|\p p]\p{1/p}$ for $p>1$ and, for every $p\geq1$, $(\Hscr\p p,\|\cdot\|_{\Hscr\p p})$ is a Banach space while $(\Hscr\p 2,\|\cdot\|_{2})$ is a Hilbert space. The space $\Hscr\p p_\mathrm{loc}$, $p\geq1$, is introduced from $\Hscr\p p$ by localization. We recall that $\Hscr\p1_\mathrm{loc}$ coincides with the space of $\Rbb$-valued local martingales (see \cite[Proposition 2.38]{J79}).

For $X\in\Hscr\p2_\mathrm{loc}$, we denote by $\aPP{X}{X}$ the $\Fbb$-predictable covariation of $X$, that is, $\aPP{X}{X}$ is the unique $\Fbb$-predictable process belonging to  $\Ascr\p+_\mathrm{loc}$ such that $X\p2-\aPP{X}{X}\in\Hscr\p1_\mathrm{loc}$.  If $X\in\Hscr\p2$, then $\aPP{X}{X}\in\Ascr\p+$ and $X\p2-\aPP{X}{X}\in\Hscr\p1$.

For any semimartingale $X$ we denote by $[X,X]$ the quadratic variation of $X$:
\[
[X,X]_t:=\aPP{X\p c}{X\p c}_t+\sum_{s\leq t}(\Delta X_s)\p2,\quad t\geq0,
\]
where $X\p c$ denotes the continuous local martingale part of $X$.

We stress that, if $X\in\Hscr\p2$, then
 $[X,X]\in\Ascr\p+$ and $[X,X]-\aPP{X}{X}\in\Hscr\p1$, i.e., $\aPP{X}{X}$ is the $\Fbb$-dual predictable projection of $[X,X]$. If $X$ is a continuous local martingale, then the identity $[X,X]=\aPP{X}{X}$ holds.

We are now going to recall the stochastic integral for a \emph{multidimensional} martingale. For this part we refer to \cite[Chapter VI, Section 4\S a]{J79}. Let $X=(X\p1,\ldots,X\p d)\p{tr}$ be an $\Rbb\p d$-valued process such that $X\p i\in\Hscr\p p_\mathrm{loc}$, $p\geq1$, $i=1,\ldots,d$, that is, $X$ is an $\Rbb\p d$-valued local martingale. We denote by $a$ and $A$ the processes given in \cite[Chapter VI, Section 4\S a]{J79} such that $[X,X]=a\cdot A$ (see \cite[Eq.\ (4.56)]{J79}). Let $K$ be an $\Rbb\p d$-valued predictable process and define $\|K\|_{\Lrm\p p(X)}=\Ebb[((K\p{tr} aK)\cdot A)\p{p/2}_\infty]$. We then define 
$
\Lrm\p p(X):=\{K \textnormal{ predictable and }\ \Rbb\p d\textnormal{-valued}:\|K\|_{\Lrm\p p(X)}<+\infty\}$. Notice that the identity $
\Lrm\p p(X)=\{K \textnormal{ predictable and }\ \Rbb\p d\textnormal{-valued}:((K\p{tr} aK)\cdot A)\p{p/2}\in\Ascr\p+\}$ obviously holds{\pa. The space $\Lrm\p p_\mathrm{loc}(X)$ is introduced making use of $\Ascr\p+_\mathrm{loc}$ instead of $\Ascr\p+$.}
 We recall that, if $X\p i\in\Hscr\p p_\mathrm{loc}$, $i=1,\ldots,d$, then $(\Lrm\p p(X),\|\cdot\|_{\Lrm\p p(X)})$ is a complete space (see \cite[Theorem 4.60]{J79}). For $K\in\Lrm\p 1_\mathrm{loc}(X)$ we denote by $K\cdot X$  the \emph{multidimensional} stochastic integral of $K$ with respect to $X$. We notice that $K\cdot X\in\Hscr\p 1_\mathrm{loc}$, the identity $[K\cdot X,K\cdot X]=K\p{tr}aK\cdot A$ holds (see \cite[Remark 4.61]{J79}) and $K\cdot X\in\Hscr\p p$ if and only if $K\in\Lrm\p p(X)$, this latter claim being a direct consequence of \cite[Eq.\ (4.57)]{J79} and of Burkh\"older--Davis--Gundy's inequality (from now on BDG's inequality), since $\Ebb[[K\cdot X,K\cdot X]_\infty\p{p/2}]=\|K\|_{\Lrm\p p(X)}$. To specify the filtration, we sometimes write $\Lrm\p p(X,\Fbb)$. Notice that, if $X\in\Hscr\p p_\mathrm{loc}$ is real-valued, then $\Lrm\p p(X)$ is the space of $\Fbb$-predictable processes $K$ such that $(K\p2\cdot [X,X])\p{p/2}\in\Ascr\p+$. Furthermore, in the special case in which $X$ is a local martingale of finite variation and $K\in\Lrm\p1_\mathrm{loc}(X)$ is such that the Stieltjes--Lebesgue integral of $K$ with respect to $X$ can be defined (as, e.g., if $K$ is locally bounded), then the stochastic integral and the Stieltjes--Lebesgue integral of $K$ with respect to $X$ coincide (see \cite[Remark 2.47]{J79}). Hence, the notation $K\cdot X$ is not ambiguous.

We shall sometimes consider a finite time horizon $T>0$ and stochastic processes restricted to the finite time interval $[0,T]$. In this case, we write $\Fbb_T=(\Fscr_t)_{t\in[0,T]}$ to denote the filtration $\Fbb$ restricted to the finite time interval $[0,T]$. We furthermore denote by $\Hscr\p p_T$ the space of $p$-integrable $\Fbb$-martingales on $[0,T]$. We write $\Hscr\p p_{\mathrm{loc}, T}$ for the localized version of $\Hscr\p p_T$. Analogously, we write $\Lrm\p p_T(X)$ for the space of the predictable integrands restricted to the finite time interval $[0,T]$ for a $d$-valued local martingale $X=(X\p1,\ldots,X\p d)\p{tr}$, $X\p i\in\Hscr\p p_{\mathrm{loc}, T}$, $i=1,\ldots,d$.

\section{Random Measures and Weak Representation Property}\label{sec:WRP}
Let $\mu$ be a nonnegative random measure on $\Rbb_+\times E$ in the sense of \cite[Definition II.1.3]{JS00}, where $E$ coincides with $\Rbb\p d$ or with a Borel subset of $\Rbb\p d$. Notice that we assume $\mu(\om,\{0\}\times E)=0$ identically.

Let $W$ be a $\Pscr(\Fbb)\otimes\Bscr(E)$-measurable mapping on the set $\Om\times\Rbb_+\times E$, where we denote by $\Pscr(\Fbb)$ the $\sig$-algebra generated by the $\Fbb$-predictable sets of $\Om\times\Rbb_+$ {\pa and by $\Bscr(E)$ the Borel $\sig$-algebra on $E$}. We say that $W$ is an $\Fbb$-predictable mapping.

Let $W$ be an $\Fbb$-predictable mapping. For $t\geq0$, {\pa following \cite[Chapter II]{JS00},} we define
\[
 W\ast\mu(\om)_t:=\begin{cases} \displaystyle\int_{[0,t]\times E}W(\om,t,x)\mu(\om,\rmd t,\rmd x),&\quad \textnormal{if } \displaystyle\int_{[0,t]\times E}|W(\om,t,x)|\mu(\om,\rmd t,\rmd x)<+\infty;\\\\
\displaystyle+\infty,&\quad\textnormal{else}.
\end{cases}
\]
We say that $\mu$ is an $\Fbb$-predictable random measure if $W\ast\mu$ is $\Fbb$-predictable for every $\Fbb$-predictable mapping $W$.

Let $X$ be an $\Rbb\p d$-valued $\Fbb$-semimartingale. We denote by $\mu\p X$ the jump measure of $X$, that is,
\[
\mu\p{X}(\om,\rmd t,\rmd x)=\sum_{s>0}1_{\{\Delta X_s(\om)\neq0\}}\delta_{(s,\Delta X_s(\om))}(\rmd t,\rmd x),
\]
where, here and in the whole paper, $\delta_a$ denotes the Dirac measure at point $a$ (which can be $d$-dimensional, $d\geq1$). 

From \cite[Theorem II.1.16]{JS00}, $\mu\p X$ is an \emph{integer-valued random measure} with respect to $\Fbb$ (see \cite[Definition II.1.13]{JS00}). By $(B\p X, C\p X,\nu\p X)$ we denote the $\Fbb$-predictable characteristics of $X$ with respect to the truncation function $h(x)=1_{\{|x|\leq1\}}x$ (see \cite[Definition II.2.3]{JS00}). We recall that $\nu\p X$ is a predictable random measure characterized by the following two properties: For any $\Fbb$-predictable mapping $W$ such that $|W|\ast\mu\p X\in\Ascr\p +_\mathrm{loc}$, we have $|W|\ast\nu\p X\in\Ascr\p +_\mathrm{loc}$ and $(W\ast\mu\p X-W\ast\nu\p X)\in\Hscr\p1_\mathrm{loc}$ {\pa(see \cite[Theorem II.1.8]{JS00})}.

We are now going to introduce the stochastic integral with respect to $(\mu\p X-\nu\p X)$  of an  $\Fbb$-predictable mapping  $W$.

Let $W$ be an $\Fbb$-predictable mapping. We define the process $\wt W$ by
\begin{equation}\label{eq:def.wr}
\wt W_t(\om):=W(\om,t,\Delta X_t(\om))1_{\{\Delta X_t(\om)\neq0\}}-\widehat W_t(\om),
\end{equation}
where, for $t\geq0$,
\[
\widehat W_t(\om):=\begin{cases} \displaystyle\int_{\Rbb\p d}W(\om,t,x)\nu\p X(\om,\{t\}\times\rmd x),&\quad \textnormal{if } \displaystyle\int_{\Rbb\p d}|W(\om,t,x)|\nu\p X(\om,\{t\}\times\rmd x)<+\infty;\\\\
\displaystyle+\infty,&\quad\textnormal{else}.
\end{cases}
\]
Notice that, according to \cite[Lemma II.1.25]{JS00}, $\widehat W$ is predictable and a version of the predictable projection of the process $(\om,t)\mapsto W(\om,t,\Delta X_t(\om))1_{\{\Delta X_t(\om)\neq0\}}$. Furthermore, since from \cite[Corollary II.1.19]{JS00} we have $\nu\p X(\om,\{t\}\times\Rbb\p d)=0$ if and only if $X$ is quasi-left continuous, we deduce that for any quasi-left continuous semimartingale $X$ and for any $\Fbb$-predictable $W$, the identity $\widehat W\equiv0$ holds.  

We introduce (see \cite[(3.62)]{J79})
\[
\textstyle\Gscr\p p(\mu\p X):=\big\{W:\ W\textnormal{ is an } \Fbb\textnormal{-predictable mapping and }\ \big(\sum_{0\leq s\leq \cdot}\wt W\p2_s\big)\p{p/2}\in \Ascr\p+\big\}.
\]
The definition of $\Gscr\p p_\mathrm{loc}(\mu\p X)$ is similar and makes use of $\Ascr\p+_\mathrm{loc}$ instead. If a finite random time horizon $T>0$ is fixed, we denote by $\Gscr\p p_T(\mu\p X)$ (resp., by $\Gscr\p p_{\mathrm{loc},T}(\mu\p X)$) the restriction of {\pa $\Gscr\p p(\mu\p X)$} (resp., of $\Gscr\p p_\mathrm{loc}(\mu\p X)$) to the finite time interval $[0,T]$. To specify the filtration we sometimes write $\Gscr\p p(\mu\p X,\Fbb)$. 
Setting
\[
\textstyle \|W\|_{\Gscr\p p(\mu\p X)}:=\Ebb\Big[\big(\sum_{s\geq 0}\wt W\p2_s\big)\p{p/2}\Big]\p{1/p},
\] 
we get a semi-norm on $\Gscr\p p(\mu\p X)$. 

Let now $W\in\Gscr\p1_\mathrm{loc}(\mu\p X)$. The stochastic integral of $W$ with respect to $(\mu\p X-\nu\p X)$ is denoted by $W\ast(\mu\p X-\nu\p X)$ and is defined as the unique purely discontinuous local martingale $Z\in\Hscr\p1_\mathrm{loc}$ such that $Z_0=0$ and $\Delta Z=\wt W$. To justify this definition, see \cite[Definition II.1.27]{JS00} and the subsequent comment. We only recall that, according to \cite[Proposition 3.66]{J79}, $W\ast(\mu\p X-\nu\p X)\in\Hscr\p p$ if and only if $W\in\Gscr\p p(\mu\p X)$.

We are now ready to give the definition of the WRP with respect to the filtration $\Fbb$ for an $\Fbb$-semimartingale $X$. 

\begin{definition}\label{def:w.prp}
For a fixed $p\geq1$, we say that the $\Rbb\p d$-valued $\Fbb$-semimartingale $X$ with continuous local martingale part $X\p c$ and characteristics $(B\p X,C\p X,\nu\p X)$ possesses the $\Hscr\p p$-WRP with respect to $\Fbb$ if every $N\in\Hscr\p p(\Fbb)$ can be represented as
\begin{equation}\label{eq:wprp}
N=N_0+K\cdot X\p c+W\ast(\mu\p X-\nu\p X),\quad K\in\Lrm\p p(X\p c,\Fbb),\quad W\in\Gscr\p p(\mu\p X,\Fbb).
\end{equation}
\end{definition}
We remark that Definition \ref{def:w.prp} is similar to \cite[Definition III.4.22]{JS00}. 

At a first look it could seem that the $\Hscr\p p$-WRP gets stronger as $p$ increases. The next proposition shows that all $\Hscr\p p$-WRP are in fact equivalent.

\begin{proposition}\label{prop:eq.wprp}
If the $\Rbb\p d$-valued semimartingale $X$ possesses the $\Hscr\p 1$-\textnormal{WRP} with respect to $\Fbb$, then it possesses the $\Hscr\p p$-\textnormal{WRP}, for every $p\geq1$. Conversely, if $X$ possesses the $\Hscr\p p$-\textnormal{WRP} with respect to $\Fbb$ for a fixed $p>1$, then it possess the $\Hscr\p1$-\textnormal{WRP} and, hence, the $\Hscr\p q$-\textnormal{WRP} for every $q\geq1$.
\end{proposition} 
\begin{proof}
Let $X$ possess the $\Hscr\p 1$-WRP with respect to $\Fbb$ and let $p>1$ be fixed. Then, every $N\in\Hscr\p p$ can be represented as in \eqref{eq:wprp} with $K\in\Lrm\p 1(X\p c)$ and $W\in\Gscr\p 1(\mu\p X)$. Therefore, denoting by $[N,N]$ the quadratic variation of $N$, by the definition of $W\ast(\mu\p X-\nu\p X)$, we get
\[
[N,N]_t=\aPP{K\cdot X\p c}{K\cdot X\p c}_t+\sum_{0\leq s\leq t}\wt W_s\p 2,\quad t\geq0.
\]
Hence, we can estimate each addend in the previous identity by $[N,N]_\infty$. Thus, both $\aPP{K\cdot X\p c}{K\cdot X\p c}\p{p/2}$ and $(\sum_{0\leq s\leq\cdot}\tilde W_s\p 2)\p{p/2}$ belong to $\Ascr\p+$, since $[N,N]\p{p/2}$ does. Therefore, $K\in\Lrm\p p(X\p c)$ and $W\in\Gscr\p p(\mu\p X)$, meaning that $X$ possesses the $\Hscr\p p$-WRP. Since $p>1$ is arbitrary, we deduce that $X$ has the $\Hscr\p p$-WRP for every $p\geq1$. Conversely, let $X$ possess the $\Hscr\p p$-WRP, for some $p>1$. We show that $X$ possesses the $\Hscr\p 1$-WRP. From this and the previous step, we deduce that $X$ has the $\Hscr\p q$-WRP for every $q\geq1$. From \cite[Proposition 2.39]{J79}, $\Hscr\p p$ is dense in $(\Hscr\p1, \|\cdot\|_{\Hscr\p1})$. Hence, for each $N\in\Hscr\p 1$ there exist $(N\p n)_n\subseteq \Hscr\p p$ such that $N\p n\longrightarrow N$ in $(\Hscr\p1, \|\cdot\|_{\Hscr\p1})$ as $n\rightarrow+\infty$ and, by assumption,
\[
N\p n=N_0\p n+K\p n\cdot X\p c+W\p n\ast(\mu\p X-\nu\p X),\quad K\p n\in\Lrm\p p(X\p c),\quad W\p n\in\Gscr\p p(\mu\p X).
\]
Using this relation, \cite[Eq.\ (4.58) and Remark 4.61(2)]{J79} and BDG's inequality, we see that $(K\p n)_n$ is a Cauchy sequence in the space $(\Lrm\p 1(X\p c),\|\cdot\|_{\Lrm\p 1(X\p c)})$ and $(W\p n)_n$ is a Cauchy sequence in the space $(\Gscr\p1(\mu\p X),\|\cdot\|_{\Gscr\p1(\mu\p X)})$. Since, from \cite[Theorem 4.60]{J79}, the space $(\Lrm\p 1(X\p c),\|\cdot\|_{\Lrm\p 1(X\p c)})$ is complete, we find $K\in\Lrm\p 1(X\p c)$ which is the limit in $\Lrm\p 1(X\p c)$ of $(K\p n)_n$. Thus, by BDG's inequality, we obtain $\|(K\p n-K)\cdot X\p c\|_{\Hscr\p1}\leq c_1\|K\p n-K\|_{\Lrm\p 1(X\p c)}\longrightarrow 0$ as $n\rightarrow+\infty$, where $c_1>0$ is a constant. Concerning $(W\p n)_n$, by BDG's inequality, there exists a constant $C_1>0$ such that
\[
\|(W\p n-W\p m)\ast(\mu\p X-\nu\p X)\|_{\Hscr\p 1}\leq C_1\|W\p n-W\p m\|_{\Gscr\p 1(\mu\p X)}.
\]
Thus, $(W\p n\ast(\mu\p X-\nu\p X))_n$ is a Cauchy sequence in $\Kscr\p1(\mu\p X):=\{W\ast(\mu\p X-\nu\p X),\ W\in\Gscr\p1(\mu\p X)\}$ which, because of \cite[Theorem 4.46]{J79}, is a closed subspace of $(\Hscr\p 1,\|\cdot\|_{\Hscr\p 1})$. So, there exists $W$ in $\Gscr\p1(\mu\p X)$ such that $W\p n\ast(\mu\p X-\nu\p X)\longrightarrow W\ast(\mu\p X-\nu\p X)$ in $(\Hscr\p 1,\|\cdot\|_{\Hscr\p 1})$ as $n\rightarrow+\infty$, which, again by BDG's inequality, implies $W\p n\longrightarrow W$ in $(\Gscr\p1(\mu\p X), \|\cdot\|_{\Gscr\p1(\mu\p X)})$ as $n\rightarrow+\infty$. Taking the limit in $\Hscr\p1$ yields
\[
N=\lim_{n\rightarrow+\infty}N\p n=\lim_{n\rightarrow+\infty}\big(N_0\p n+K\p n\cdot X\p c+W\p n\ast(\mu\p X-\nu\p X)\big)=N_0+K\cdot X\p c+W\ast(\mu\p X-\nu\p X)
\]
and the proof of the lemma is complete.
\end{proof}
According to Proposition \ref{prop:eq.wprp}, if there is no need to specify the involved martingale space $\Hscr\p p$, $p\geq1$, we will sometime simply say WRP instead of $\Hscr\p p$-WRP.  Furthermore, to verify that an $\Rbb\p d$-valued semimartingale $X$ possesses the $\Hscr\p1$-WRP with respect to $\Fbb$, it is enough to check that $X$ possesses the $\Hscr\p 2$-WRP with respect to $\Fbb$. For later use, the following lemma will be useful:
\begin{lemma}\label{lem:WRP.lim}
Let $(\xi\p n)_n\subseteq L\p2(\Om,\Fscr_\infty,\Pbb)$, $\xi\p n\longrightarrow\xi$ in $L\p2(\Om,\Fscr_\infty,\Pbb)$ as $n\rightarrow+\infty$. If 
\[
\xi\p n=\Ebb[\xi\p n|\Fscr_0]+K\p n\cdot X\p c_\infty+W\p n\ast(\mu\p X-\nu\p X)_\infty,\quad K\p n\in\Lrm\p2(X\p c,\Fbb),\quad W\p n\in\Gscr\p2(\mu\p X,\Fbb),
\]
then there exist $K\in\Lrm\p2(X\p c)$ and $W\in\Gscr\p2(\mu\p X)$ such that
\[
\xi=\Ebb[\xi|\Fscr_0]+K\cdot X\p c_\infty+W\ast(\mu\p X-\nu\p X)_\infty.
\]
\end{lemma} 
The proof of Lemma \ref{lem:WRP.lim} can be given in a similar fashion as the one of Proposition \ref{prop:eq.wprp} and is, therefore, omitted.

In the next remark we compare the WRP with another representation property for local martingales, that is, the \emph{predictable representation property} (from now on PRP).

{\pa For a process $X$, we denote in this paper by $\Fbb\p X$ the smallest filtration satisfying the usual conditions such that $X$ is adapted.}
\begin{remark}[WRP and PRP]\label{rem:wrp.prp}
Let $X\in\Hscr\p1_\mathrm{loc}$ be an $\Fbb$-local martingale. We recall that $X$ possesses the PRP with respect to $\Fbb$ if every $\Fbb$-local martingale $N$ can be represented as
\[
N=N_0+H\cdot X,\quad H\in\Lrm\p1_\mathrm{loc}(X,\Fbb).
\]

From \cite[Theorem 13.14]{HWY92}, we know that the PRP implies the WRP. The converse is, in general, not true. As a counter-example we recall the case of a martingale $X$ which is a L\'evy process: In this case $X$ possesses the WRP {\pa with respect to $\Fbb\p X$} but it possesses the PRP {\pa with respect to $\Fbb\p X$} if and only if $X$ is a Brownian motion or a compensated Poisson process (see \cite[Corollary 13.54]{HWY92}).
\end{remark}

\begin{remark}\label{rem:wrp}
{\pa We now list some examples from the literature of semimartingales possessing the WRP (and hence the $\Hscr\p p$-WRP, for every $p\geq1$). }
\begin{itemize}
\item From \cite[Theorem 13.14]{HWY92}, every local martingale $X$ with the PRP (predictable representation property) possess also the WRP. Classical examples are therefore obtained assuming that $X$ is a Brownian motion or a compensated Poisson process with respect to the filtration $\Fbb\p X$.  
\item A less classical situation is the case in which $X=(X\p1,\ldots,X\p d)$ is a $d$-dimensional continuous local martingale such that $\aPP{X\p i}{X\p j}$ is deterministic and $X\p i_0=0$, $i,j=1,\ldots,d$. According to \cite[Theorem 7.1]{DTE16}, we deduce that $X$ has the PRP with respect to $\Fbb\p X_T$, for an arbitrary but fixed $T>0$ (that is, the stable subspace generated by $X$ in $\Hscr\p2_T(\Fbb\p X)$ equals $\Hscr\p2_T(\Fbb\p X)$). A special example is given by a $d$-dimensional Brownian motion $X$. Notice that in this latter case $\aPP{X\p i}{X\p j}=0$ holds, $i,j=1,\ldots,d$ $i\leq j$.

\item In \cite{E89}, Emery studied the \emph{chaotic representation property} of the Az\'ema martingales. This is a special class of square integrable martingales, obtained as solutions of a particular structure equation. An Az\'ema martingale $X$ has, in particular, the following property: For every $t\geq0$, it holds $\aPP{X}{X}_t=ct$, $c>0${\pa, that is, Az\'ema martingales are \emph{normal martingales}}. Notice that Az\'ema martingales have not, in general, independent increments (see \cite[p.\ 79]{E89} for an example). The Brownian motion and the compensated Poisson processes are examples of Az\'ema martingales. In \cite[Proposition 6]{E89}, Emery proved that some Az\'ema martingales $X$ possess the chaotic (and hence the predictable and the weak) representation property with respect to $\Fbb\p X$. 
\item In \cite{J75}, Jacod proved that if $\mu$ is a multivariate point process and $\nu$ its compensator with respect to the filtration $\Fbb\p\mu$, that is, the smallest right-continuous filtration with respect to which $\mu$ is an optional random measure (see \cite[(A.1), p.\ 36]{J75}), then every $\Fbb\p\mu$-local martingale can be represented as a stochastic integral with respect to $\mu-\nu$ (see \cite[Theorem 5.4]{J75}). Using this result, one can prove that any step process $X$ (see \cite[11.55]{HWY92}) possesses the WRP with respect to $\Fbb\p X$ (see \cite[Theorem 13.19]{HWY92}). It can be shown that, if $B$ is a Brownian motion with respect to $\Fbb\p B$ and $X$ is a step process independent of $B$, then the semimartingale $Y=B+X$ has the WRP with respect to $\Fbb\p Y$ (see \cite[Corollary 2]{WG82}).

\item In \cite{KW67}, Kunita and Watanabe established in the example on p.\ 227 and in Proposition 5.2 the WRP for a L\'evy processes $X$ with respect to $\Fbb\p X$ (see also \cite[Theorem 1.1]{K04}). 

\item In \cite[Theorem III.4.34]{JS00} the WRP has been obtained for a semimartingale $X$ with conditionally independent (and not necessarily homogeneous) increments, with respect to the right-continuous filtration $\Fbb$ generated by $\Fbb\p X$ and by an initial $\sig$-field $\Hscr$ (see \cite[III.2.12]{JS00}). 
\end{itemize}
\end{remark}

\section{Progressively Enlarged Filtrations: A Brief Summary}\label{sec:prog.en}

We denote by $\tau$ a $(0,+\infty]$-valued random variable and call $\tau$ a \emph{random time}. The \emph{default process} $H=(H_t)_{t\geq0}$ associated with $\tau$ is defined by $H_t(\om):=1_{[\tau,+\infty)}(\om,t)$. The filtration generated by $H$ is denoted by $\Hbb=(\Hscr_t)_{t\geq0}$. We stress that $H_0=0$ so that $H$ is an increasing process in the sense of \cite[Definition I.3.1]{JS00}.

For a filtration $\Fbb=(\Fscr_t)_{t\geq0}$ satisfying the usual conditions, $\Gbb=(\Gscr_t)_{t\geq0}$ denotes the \emph{progressive enlargement of} $\Fbb$ \emph{by} $\tau$ and it is defined by

\[
 \Gscr_t:=\bigcap_{\ep>0} \wt\Gscr_{t+\ep},\quad t\geq0,
\] 
where $\wt\Gscr_t=\Fscr_t\vee\Hscr_t$, $t\geq0$, and $\wt \Gbb=(\wt\Gscr_t)_{t\geq0}$.
That is, $\Gbb$ is the \emph{smallest} right-continuous (and, hence, satisfying the usual conditions) filtration containing $\Fbb$ and such that $\tau$ is a $\Gbb$-stopping time.
 
\begin{definition}\label{def:av.im} Let $\tau:\Om\longrightarrow(0,+\infty]$ be a random time.

(i) We say that $\tau$ satisfies hypothesis $(\Ascr)$, if $\tau$ avoids $\Fbb$-stopping times, that is, if for every $\Fbb$-stopping time $\sig$, $\Pbb[\tau=\sig<+\infty]=0$ holds. 

(ii) We say that $\tau$ satisfies hypothesis $(\Hscr)$, if $\Fbb$ is immersed in $\Gbb$, that is, if $\Fbb$-martingales remain $\Gbb$-martingales.
\end{definition}
 
 The following assumptions will play a key role in this paper:

\begin{assumption}\label{ass:av.im}
The random time $\tau$ satisfies both hypotheses $(\Ascr)$ and $(\Hscr)$ (see Definition \ref{def:av.im}).
\end{assumption}

We now shortly comment Assumptions \ref{ass:av.im}. Hypothesis $(\Ascr)$ is widely used in the literature about progressively enlarged filtrations, especially if $\Fbb$-martingales are not all continuous. In this paper, which deals with the propagation of the WRP to the filtration $\Gbb$, it is natural (and convenient) to assume  hypothesis $(\Ascr)$. Indeed, if $(\Ascr)$ is not satisfied, as a limit case, it can happen that $\tau$ is an $\Fbb$ stopping time: Hence, $\Fbb$ and $\Gbb$ coincide.  To role out this trivial case, we require $(\Ascr)$. The interpretation of $(\Ascr)$ is also very natural: We are enlarging $\Fbb$ adding some completely new information, which is not contained in $\Fbb$. 

It is well known (see \cite[Theorem 3.2 (a)]{AJ17}) that hypothesis $(\Hscr)$ is equivalent to the conditional independence of $\Fscr_\infty$ and $\Gscr_t$ given $\Fscr_t$. Therefore, it is obvious that hypothesis $(\Hscr)$ is satisfied if $\tau$ is independent of $\Fbb$. Another important case, especially for applications to credit risk, in which the immersion property is satisfied, is when the random time $\tau$ is obtained by the \emph{Cox construction} (see \cite[Section 2.3 and Lemma 2.28]{AJ17}). 
\\[.5em]
\indent The next lemma, which will be useful later, exhibits a generating system for the $\sig$-algebra $\wt\Gscr_t$, for every $t\geq0$.
\begin{lemma}\label{lem:gen.sys} For every $t\geq0$, let us introduce the system
\[
\Cscr_t:=\{\xi\in\wt\Gscr_t:\ \xi=\zt(1-H_s),\quad \zt\ \textnormal{ bounded and } \Fscr_t\textnormal{-measurable},\ 0\leq s\leq t\}.
\]
Then, the $\sig$-algebra generated by $\Cscr_t$, denoted by $\sig(\Cscr_t)$, coincides with $\wt\Gscr_t$, for every $t\geq0$.
\end{lemma}
\begin{proof}
Notice that the inclusion $\sig(\Cscr_t)\subseteq\wt\Gscr_t$ obviously holds. It is therefore enough to verify the converse inclusion  $\wt\Gscr_t\subseteq\sig(\Cscr_t)$. To this goal, it is sufficient to prove that each $A\in\Fscr_t$ is $\sig(\Cscr_t)$-measurable and that $H_s$ is $\sig(\Cscr_t)$-measurable, for every $0\leq s\leq t$. Since $1_A=1_A(1-H_0)$ for every $A\in\Fscr_t$, we have $\Fscr_t\subseteq\sig(\Cscr_t)$, for every $t\geq0$. Furthermore, we have $(1-H_s)=1(1-H_s)$ for every $0\leq s\leq t$. So, $H_s$ is $\sig(\Cscr_t)$-measurable for every $0\leq s\leq t$ and hence the inclusion $\wt\Gscr_t\subseteq\sig(\Cscr_t)$ holds. The proof is complete.
\end{proof}
The assumption that $\tau$ takes value in $(0,+\infty]$ simplifies the proof of Lemma \ref{lem:gen.sys}: If we only have $\Pbb[\tau=0]=0$ (as it happens if, e.g., hypothesis $(\Ascr)$ is satisfied {\pa and $\tau$ takes values in $[0,+\infty]$}), then Lemma \ref{lem:gen.sys} holds replacing $\sig(\Cscr_t)$ by $\sig(\Cscr_t)\vee\Nscr$, $\Nscr$ denoting the family of the $\Pbb$-null sets in $\Fscr$.
\\[.5em]
\indent In this paper we denote by $A$ the $\Fbb$-optional projection of the the right-continuous and bounded process $(1-H)$. Then $A$ is \cadlag, $\Fbb$-adapted and satisfies
\begin{equation}\label{eq:def.az.supm}
A_t=\Pbb[\tau>t|\Fscr_t],\quad\textnormal{a.s.},\quad t\geq0.
\end{equation}
In particular, being bounded, $A$ is a supermartingale of \emph{class} $(D)$ with respect to $\Fbb$, called \emph{the Az\'ema supermartingale}, where we recall that an $\Fbb$-adapted process $X$ is of \emph{class} (D) with respect to $\Fbb$ if the family $\{X_\sig: \sig\ \textnormal{ finite-valued } \Fbb\textnormal{-stopping time}\}$ is uniformly integrable. 

Notice that $A_0=1$. From classical  literature on martingale theory (see, e.g., \cite[Theorem 2.62 and p.\ 63]{HWY92}), it follows that $A_->0$ on $[0,\tau]$ and that $A>0$ on $[0,\tau)$. 
 
We denote by $H\p{o,\Fbb}$ the $\Fbb$-dual optional projection and  by $H\p{p,\Fbb}$ the $\Fbb$-dual predictable projection (see \cite[Theorem V.28]{D72}) of $H$. 

{\pa For the formulation of the next result, we refer to \cite[Lemma 1.48]{AJ17}.}
\begin{lemma}\label{lem:con.dop}
Let $\tau$ satisfy hypothesis $(\Ascr)$. Then,  $H\p{o,\Fbb}$ is continuous and  $H\p{p,\Fbb}=H\p{o,\Fbb}$.
\end{lemma}
\begin{proof}
If $H\p{o,\Fbb}$ is continuous, then it is $\Fbb$-predictable, $H\p{o,\Fbb}$ being $\Fbb$-adapted. Hence, the identity $H\p{p,\Fbb}=H\p{o,\Fbb}$ holds. We now come to the continuity of $H\p{o,\Fbb}$. Let $\sig$ be a finite-valued $\Fbb$-stopping time. Then, $\Delta H_\sig\p{o,\Fbb}=1_{[0,\sig]}\cdot H\p{o,\Fbb}_\infty-1_{[0,\sig)}\cdot H\p{o,\Fbb}_\infty$ and $1_{[0,\sig)}$ is an $\Fbb$-optional process. The properties of the $\Fbb$-dual optional projection and hypothesis $(\Ascr)$ now yield $\Ebb[\Delta H\p{o,\Fbb}_\sig]=\Ebb[\Delta H_\sig]=\Pbb[\tau=\sig]=0$. Hence, since $H\p{o,\Fbb}$ is increasing, we get $H\p{o,\Fbb}_\sig=H\p{o,\Fbb}_{\sig-}$ a.s.\ for every {\pa finite-valued} $\Fbb$-stopping time $\sig$. By the optional section theorem (see \cite[Theorem IV.13]{D72}), we obtain the continuity of $H\p{o,\Fbb}$. The proof is complete.
\end{proof}

We now discuss some properties of the Az\'ema supermartingale $A$, which we shall need later.
 
\begin{proposition}\label{prop:con.Z} Let $\tau$ satisfy Assumptions \ref{ass:av.im}. Then $A$ is decreasing, continuous and  $A>0$ on $[0,\tau]$.
\end{proposition}
\begin{proof}
By hypothesis $(\Hscr)$, we have $A_t=\Pbb[\tau>t|\Fscr_\infty]$ (see \cite[Theorem 3.2(c)]{AJ17}) from which we see that $A$ is decreasing. The continuity of $A$ follows by Lemma \ref{lem:con.dop} because, by \cite[Proposition 3.9(a)]{AJ17}, the Doob--Meyer decomposition of $A$ is $A=1-H\p{o,\Fbb}$. To see the last part, we observe that $A_->0$ on $[0,\tau]$ and $A=A_-$. The proof is complete.
\end{proof}

Since $H\in\Ascr\p+(\Gbb)$, $H$ being bounded, there exists the $\Gbb$-dual predictable projection $\Lm\p\Gbb$ of $H$, that is, $\Lm\p\Gbb$ is the unique $\Gbb$-predictable process in $\Ascr\p+(\Gbb)$ such that $M=(M_t)_{t\geq0}$, defined by
\begin{equation}\label{eq:def.M}
M_t:=H_t-\Lm\p\Gbb_t,\quad t\geq0,
\end{equation}
is a local martingale (see \cite[Theorem I.3.17]{JS00}). Since $H$ belongs to $\Ascr\p+(\Gbb)$, we also have that $\Lm\p\Gbb$ belongs to $\Ascr\p+(\Gbb)$. Hence,  we have $\Ebb[\sup_{t\geq0}|M_t|]\leq1+\Ebb[\Lm\p\Gbb_\infty]<+\infty$ which implies that $M$ is a true martingale and that it belongs to $\Hscr\p1(\Gbb)$. From \cite[Proposition 2.15]{AJ17}, we have
\begin{equation}\label{eq:G.com.gen}
\Lm\p\Gbb_t=\int_0\p{\tau\wedge t}\frac{1}{A_{s-}}\,\rmd H\p {p,\Fbb}_s,\quad t\geq0.
\end{equation}

By $U:=\Escr(-M)$ we denote the stochastic exponential (see \cite[Chapter I, Section 4f]{JS00}) of the martingale $-M$. The following theorem summarizes some relevant properties of $M$, $\Lm\p\Gbb$ and $U$.
\begin{theorem}\label{thm:G.com.AH}
Let $\tau$ satisfy Assumptions \ref{ass:av.im}.

\textnormal{(i)} $\Lm\p\Gbb$ is continuous and $\Lm_t\p\Gbb=-\log( A_{\tau\wedge t})$, $t\geq0$, where we define $\log(0):=-\infty$.

\textnormal{(ii)} $M\in\Hscr\p2(\Gbb)$, $M_0=0$ and $M\p2-\Lm\p{\Gbb}\in\Hscr\p1(\Gbb)$, that is, $\pb{M}{M}=\Lm\p{\Gbb}$.

\textnormal{(iii)} $U\in\Hscr\p2_\mathrm{loc}(\Gbb)$ and $U=A\p{-1}1_{[0,\tau)}$.
\end{theorem}
\begin{proof} We verify (i). By Assumptions \ref{ass:av.im}, the identity $H\p{p,\Fbb}=1-A$ holds (see \cite[Theorem 3.2(c)]{AJ17} and Lemma \ref{lem:con.dop}). Hence, form Proposition \ref{prop:con.Z} and \eqref{eq:G.com.gen} we deduce (i). We now come to (ii). Clearly, $M_0=0$. The continuity of $\Lm\p\Gbb$ yields the identity $[M,M]=H$. Thus, from \cite[Theorem 7.32]{HWY92}, we deduce $M\in\Hscr\p2(\Gbb)$. Integration by parts implies $M\p2=2M_-\cdot M+H$. Since the stochastic integral in this identity is a $\Gbb$-local martingale, $M\p 2-\aPP{M}{M}$ is a $\Gbb$-local martingale if and only if $H-\aPP{M}{M}$ is a $\Gbb$-local martingale. By the uniqueness of the $\Gbb$-dual predictable projection of an increasing process, we deduce the identity $\pb{M}{M}=\Lm\p{\Gbb}$, which completes the proof of (ii). We now verify (iii). Because of the Dol\'eans--Dade equation, we have $U=1-U_-\cdot M$. So, from (ii), we deduce $U\in\Hscr\p2_\mathrm{loc}(\Gbb)$. From Proposition \ref{prop:con.Z}, we have $A>0$ on $[0,\tau]$ so that $A\p{-1}$ is well-defined over $[0,\tau]$. Using the Dol\'eans--Dade exponential formula, we compute
\[
U_t:=\Escr(-M)_t=\exp(-M_t)\prod_{0\leq s\leq t}(1-\Delta M_s)\exp(\Delta M_s)=\begin{cases}\exp(-\log( A_{t\wedge\tau})),\quad &\textnormal{on}\quad \{t<\tau\},\\
0,\quad&\textnormal{else}.
\end{cases}
\]
The proof of the theorem is complete.
\end{proof}

\section{The Weak Representation Property in the Progressive Enlargement}\label{sec:wea.prp.G}

In this section we consider an  $\Rbb\p d$-valued semimartingale $X$ with respect to the filtration $\Fbb$. The jump-measure of $X$ is $\mu\p X$ and $(B\p X,C\p X,\nu\p X)$ is the triplet of the semimartingale characteristics of $X$. We assume that $X$ possesses the WRP with respect to $\Fbb$. In particular, according to Proposition \ref{prop:eq.wprp}, every $N\in\Hscr\p2(\Fbb)$ can be represented as
\[
N=N_0+K\cdot X\p c+W\ast(\mu\p X-\nu\p X),\quad K\in\Lrm\p2(X\p c,\Fbb),\quad W\in\Gscr\p2(\mu\p X,\Fbb).
\]
We denote by $\tau$ a random time on $\Om$ with values in $(0,+\infty]$ and by $\Gbb$ the progressive enlargement of $\Fbb$ by $\tau$. 

\begin{lemma}\label{leb:XcharG}
Let $\tau$ satisfy hypothesis $(\Hscr)$. Then $X$ is a $\Gbb$-semimartingale and the $\Gbb$-semimartingale characteristics of $X$ are again $(B\p X,C\p X,\nu\p X)$.
\end{lemma}
\begin{proof}
 Because of hypothesis $(\Hscr)$, $X$ is a semimartingale also in the filtration $\Gbb$. To compute the $\Gbb$-semimartingale characteristics of $X$ we apply hypothesis $(\Hscr)$ and
\cite[Theorem II.2.21]{JS00}. The proof is complete.
\end{proof}
Let $H=1_{[\tau,+\infty)}$ be the default process associated with $\tau$. We set $E:=\Rbb\p d\times\{0,1\}\subseteq\Rbb\p{d+1}$. If $\tau$ satisfies hypothesis $(\Hscr)$, the process $(X,H)\p{tr}$ is an $E$-valued $\Gbb$-semimartingale. To prove Theorem \ref{thm:wea.prp} below, we need the $\Gbb$-semimartingale characteristics $(B\p{(X,H)},C\p{(X,H)}, \nu\p{(X,H)})$ of the $E$-valued $\Gbb$-semimartingale $(X,H)\p{tr}$. The next step is devoted to the computation of these semimartingale characteristics. 

We denote by $x_1$ a vector of $\Rbb\p d$, while $x_2$ is a real number. We consider the $\Gbb$-semimartingale characteristics of $(X,H)\p{tr}$ with respect to the $\Rbb\p{d+1}$-valued truncation function \[h(x_1,x_2)=(h_1(x_1),h_2(x_2)):=(1_{\{|x_1|\leq1\}}x_1,1_{\{|x_2|\leq1\}}x_2).\] 

The jump measure $\mu\p H$ of $H$ is given by $\mu\p H(\om,\rmd t,\rmd x_2)=\rmd H_t(\om)\delta_1(\rmd x_2)$ and the compensator of $\mu\p H$ is $\nu\p H(\om,\rmd t,\rmd x_2)=\rmd \Lm\p\Gbb_t(\om)\delta_1(\rmd x_2)$. Hence, the $\Gbb$-semimartingale characteristics of $H$ with respect to the truncation function $h_2$ are $(\Lm\p\Gbb,0,\nu\p H)$.

Let hypothesis $(\Hscr)$ be in force. The second $\Gbb$-semimartingale characteristics of ${(X,H)}\p{tr}$ is completely determined by $C\p X$, since $C\p H=0$. Furthermore, we have $B\p{(X,H)}=(B\p X,\Lm\p\Gbb)\p{tr}$.
In summary, to determine $(B\p{(X,H)},C\p{(X,H)}, \nu\p{(X,H)})$, it is enough to compute the jump measure $\mu\p {(X,H)}$ of ${(X,H)}\p{tr}$ and then the $\Gbb$-predictable compensator $\nu\p {(X,H)}$ of $\mu\p {(X,H)}$: The next proposition is devoted to this goal.

\begin{proposition}\label{prop:ju.mea.Y} Let $\tau$ satisfy Assumptions \ref{ass:av.im}.

\textnormal{(i)}
The jump-measure $\mu\p {(X,H)}$ of the $E$-valued $\Gbb$-semimartingale ${(X,H)}\p{tr}$ is given by
\[
\mu\p {(X,H)}(\om,\rmd t,\rmd x_1,\rmd x_2)=\mu\p X(\om,\rmd t,\rmd x_1)\delta_0(\rmd x_2)+\rmd H_t(\om)\delta_1(\rmd x_2)\delta_0(\rmd x_1)\,.
\]

\textnormal{(ii)} The $\Gbb$-predictable compensator $\nu\p {(X,H)}$ of $\mu\p {(X,H)}$ is given by
{\pa\begin{equation}\label{eq:com.ans}
\nu\p {(X,H)}(\om,\rmd t,\rmd x_1,\rmd x_2)=\nu\p X(\om,\rmd t,\rmd x_1)\delta_0(\rmd x_2)+\rmd \Lm_t\p\Gbb(\om)\delta_1(\rmd x_2)\delta_0(\rmd x_1).
\end{equation}}
\end{proposition}

\begin{proof} We start proving (i). First, we observe that $X$ {\pa is a $\Gbb$-semimartingale, because of hypothesis $(\Hscr)$. Therefore, $(X,H)\p{tr}$ is a $\Gbb$-semimartingale.} By definition, the jump-measure $\mu\p {(X,H)}$ of {\pa${(X,H)\p{tr}}$} is given by
\[
\mu\p {(X,H)}(\om,\rmd t,\rmd x_1,\rmd x_2)=\sum_{s>0}1_{\{\Delta {(X,H)}_s(\om)\neq0\}}\delta_{(s,\Delta {(X,H)}_s(\om))}(\rmd t,\rmd x_1,\rmd x_2).
\]
Since $\tau$ satisfies hypothesis $(\Ascr)$, the $\Gbb$-semimartingales $X$ and $H$ have no common jumps. Therefore, the identity
$\{\Delta X\neq0\}\cap\{\Delta H\neq0\}=\emptyset$ holds {\pa and we} have the inclusions $\{\Delta X\neq0\}\subseteq\{\Delta H=0\}$ and
$\{\Delta H\neq0\}\subseteq\{\Delta X=0\}$. Thus,  
\[
\begin{split}
\mu\p {(X,H)}(\om,\rmd t,\rmd x_1,\rmd x_2)&=\sum_{s>0}1_{\{\Delta X_s(\om)\neq0\}\cap\{ \Delta H_s(\om)=0\}}\delta_{(s,\Delta X_s(\om))}(\rmd t,\rmd x_1)\delta_0(\rmd x_2)\\&\qquad+\sum_{s>0}1_{\{\Delta X_s(\om)=0\}\cap\{ \Delta H_s(\om)\neq0\}}\delta_{(s,\Delta H_s(\om))}(\rmd t,\rmd x_2)\delta_0(\rmd x_1)
\\&=\left(\sum_{s>0}1_{\{\Delta X_s(\om)\neq0\}}\delta_{(s,\Delta X_s(\om))}(\rmd t,\rmd x_1)\right)\delta_0(\rmd x_2)\\&\qquad+\left(\sum_{s>0}1_{\{ \Delta H_s(\om)\neq0\}}\delta_{(s,\Delta H_s(\om))}(\rmd t,\rmd x_2)\right)\delta_0(\rmd x_1)
\\&=\mu\p X(\om,\rmd t,\rmd x_1)\delta_0(\rmd x_2)+\mu\p H(\om,\rmd t,\rmd x_2)\delta_0(\rmd x_1)
\\&=\mu\p X(\om,\rmd t,\rmd x_1)\delta_0(\rmd x_2)+\rmd H_t(\om)\delta_1(\rmd x_2)\delta_0(\rmd x_1)
\end{split}
\]
and the proof of (i) is complete. We now show (ii). To see that the $\Gbb$-predictable random measure $\nu\p {(X,H)}$ is the $\Gbb$-predictable compensator of $\mu\p {(X,H)}$, it is enough to verify $|W|\ast\nu\p {(X,H)}\in\Ascr\p+_\mathrm{loc}(\Gbb)$ and $(W\ast\mu\p {(X,H)}-W\ast\nu\p {(X,H)})\in\Hscr\p1_\mathrm{loc}(\Gbb)$, for all $\Gbb$-predictable mappings $W$on the space $E=\Rbb\p d\times\{0,1\}$ such that $|W|\ast\mu\p {(X,H)}$ belongs to $\Ascr\p+_\mathrm{loc}(\Gbb)$ (see \cite[Theorem II.1.8]{JS00}). Let now $W\ast\mu\p {(X,H)}\in\Ascr\p+_\mathrm{loc}(\Gbb)$, where $W\geq0$ is a $\Gbb$-predictable mapping. We define $W\p0(t,x_1):=W(t,x_1,0)$ and $W\p{0,1}(t):=W(t,0,1)$. Then $W\p 0$ is a $\Gbb$-predictable mapping, $W\p{0,1}$ is a $\Gbb$-predictable process and, from (i), we have:
\[
W\ast\mu\p {(X,H)}_t=W\p 0\ast\mu\p X_t+W\p{0,1}\cdot H_t, \quad t\geq0,
\]
meaning that $W\p{0}\ast\mu\p X\in\Ascr\p+_\mathrm{loc}(\Gbb)$ and $W\p{0,1}\cdot H\in\Ascr\p+_\mathrm{loc}(\Gbb)$, because of $W\geq0$. Since $\Lm\p\Gbb$ is the $\Gbb$-predictable compensator of $\mu\p H$ and since, because of Lemma \ref{leb:XcharG}, $\nu\p X$ is the $\Gbb$-predictable compensator of $\mu\p X$, this yields $W\p{0,1}\cdot \Lm\p\Gbb\in\Ascr\p+_\mathrm{loc}(\Gbb)$, $W\p{0}\ast\nu\p X\in\Ascr\p+_\mathrm{loc}(\Gbb)$,  $(W\p{0,1}\cdot H-W\p{0,1}\cdot \Lm\p\Gbb)\in\Hscr\p1_\mathrm{loc}(\Gbb)$ and $(W\p{0}\ast\mu\p X-W\p{0}\ast\nu\p X)\in\Hscr\p1_\mathrm{loc}(\Gbb)$. Hence, {\pa denoting by $\wt\nu\p {(X,H)}$ the predictable measure defined on the right-hand side of \eqref{eq:com.ans}, we have}
\[
{\pa W\ast\wt\nu\p {(X,H)}}=W\p{0}\ast\nu\p X+W\p{0,1}\cdot \Lm\p\Gbb\in\Ascr\p+_\mathrm{loc}
\]
and
\[
W\ast\mu\p {(X,H)}-{\pa W\ast\wt\nu\p {(X,H)}}=(W\p 0\ast\mu\p X-W\p{0}\ast\nu\p X)+(W\p{0,1}\cdot H-W\p{0,1}\cdot \Lm\p\Gbb)\in\Hscr\p1_\mathrm{loc}(\Gbb).
\] 
 If $W$ is an arbitrary $\Gbb$-predictable mapping and $|W|\ast\mu\p {(X,H)}\in\Ascr\p+_\mathrm{loc}(\Gbb)$, from the previous step, we deduce ${\pa|W|\ast\wt\nu\p {(X,H)}}\in\Ascr\p+_\mathrm{loc}$ and, applying the previous step to the positive and negative part of $W$, we additionally obtain $(W\ast\mu\p {(X,H)}-{\pa W\ast\nu\p {(X,H)}})\in\Hscr\p1_\mathrm{loc}(\Gbb)$. {\pa By \cite[Theorem I.1.8]{JS00}, the $\Gbb$-predictable compensator $\nu\p {(X,H)}$ of $\mu\p {(X,H)}$ coincides with $\wt\nu\p{(X,H)}$} and the proof of the proposition is complete. 
\end{proof}

We now consider an arbitrary but fixed time horizon $T>0$. We recall the notation $\Gbb_T=(\Gscr_t)_{t\in[0,T]}$ and $\Hscr\p2_T(\Gbb)$ is the space of square integrable $\Gbb$-adapted martingales restricted to the finite time interval $[0,T]$. Analogously $\Lrm\p2_T(X\p c,\Gbb)$ and $\Gscr\p2_T(\mu\p{(X,H)},\Gbb)$ are the spaces of $\Gbb$-predictable integrands for $X\p c$ and $\mu\p{(X,H)}-\nu\p{(X,H)}$, respectively, restricted to the finite time interval $[0,T]$. 

The following theorem, which is the main results of this paper, shows that the WRP of $X$ with respect to $\Fbb$ propagates to the semimartingale $(X,H)\p{tr}$ in the filtration $\Gbb_T$.

\begin{theorem}\label{thm:wea.prp}
Let $\tau$ satisfy Assumptions \ref{ass:av.im} and let $X$ be an $\Fbb$-semimartingale possessing the \textnormal{WRP} with respect to $\Fbb$. Let $T>0$ be an arbitrary but fixed time horizon. Then the $\Gbb$-semimartingale ${(X,H)}\p{tr}$ possesses the $\Hscr\p2$-\emph{WRP} with respect to $\Gbb_T$, that is, every $N\in\Hscr\p2_T(\Gbb)$ can be represented as
\begin{equation}\label{eq:rep.G}
N_t=N_0+K\cdot X\p c_t+W\ast(\mu\p{(X,H)}-\nu\p{(X,H)})_t,\quad t\in[0,T],\ K\in\Lrm_T\p2(X\p c,\Gbb),\ W\in\Gscr\p2_T(\mu\p{(X,H)},\Gbb).
\end{equation}
Furthermore, the representation \eqref{eq:rep.G} is unique on $[0,T]$ in the following sense: If $K\p\prime\in\Lrm_T\p 2(X\p c,\Gbb)$ and $W\p\prime\in\Gscr\p 2_T(\mu\p{(X,H)},\Gbb)$ are other integrands such that \eqref{eq:rep.G} holds, we then have $\|K-K\p\prime\|_{\Lrm_T\p p(X\p c,\Gbb)}=0$, $\|W-W\p\prime\|_{\Gscr\p p_T(\mu\p{(X,H)},\Gbb)}=0$ and $K\cdot X\p c_t=K\p\prime\cdot X\p c_t$, $W\ast(\mu\p{(X,H)}-\nu\p{(X,H)})_t=W\p\prime\ast(\mu\p{(X,H)}-\nu\p{(X,H)})_t$, for all $t\in[0,T]$ a.s.
\end{theorem}
\begin{remark}\label{comm:ass}
A seminal paper about the propagation of the PRP is Kusuoka \cite{Ku99}. In \cite{Ku99}, Kusuoka assumed that $\Fbb$ is the filtration generated by a Brownian motion and he considered a finite-valued random time $\tau$. Furthermore, Kusuoka assumed that $\Fbb$ is immersed in $\Gbb$ and that the $\Gbb$-predictable compensator $\Lm\p\Gbb$ of the default process $H$ is absolutely continuous with respect to the Lebesgue measure. This latter assumption implies, in particular, that $\tau$ avoids finite-valued $\Fbb$-stopping times. Indeed, let $\Lm\p\Gbb$ be absolutely continuous with respect to the Lebesgue measure. Then $\Lm\p\Gbb$ is, in particular, continuous. If now $\sig$ is a finite-valued $\Fbb$-stopping time, both $1_{[0,\sig]}$ and $1_{[0,\sig)}$ are $\Fbb$-predictable (and hence $\Gbb$-predictable) processes, $\Fbb$ being the Brownian filtration. Therefore, using the properties of the $\Gbb$-dual predictable projection, we get the identity $\Pbb[\tau=\sig]=\Ebb[\Delta H_\sig]=\Ebb[\Delta \Lm\p\Gbb_\sig]=0$, meaning that $\tau$ avoids finite-valued $\Fbb$-stopping times. Hence, for a finite valued random time $\tau$, Assumptions \ref{ass:av.im} are weaker than those in \cite{Ku99}. Therefore, since in the Brownian case the WRP and the PRP coincides, Theorem \ref{thm:wea.prp} is a generalization of the result about the PRP obtained in \cite{Ku99}.
\end{remark}

We now come to the proof of Theorem \ref{thm:wea.prp}.

\begin{proof}[Proof of Theorem \ref{thm:wea.prp}]
We first discuss the uniqueness of the representation \eqref{eq:rep.G}. We observe that the stochastic integral with respect to $X\p c$ defines an isometry between the spaces $(\Hscr\p2_T(\Gbb),\|\cdot\|_2)$ and $(\Lrm\p2_T(X\p c,\Gbb),\|\cdot\|_{\Lrm_T\p2(X\p c,\Gbb)})$ and the stochastic integral  with respect to $\mu\p{(X,H)}-\nu\p{(X,H)}$ defines an isometry between the spaces $(\Hscr\p2_T(\Gbb),\|\cdot\|_2)$ and $(\Gscr_T\p2(\mu\p{(X,H)},\Gbb),\|\cdot\|_{\Gscr\p2_T(\mu\p{(X,H)},\Gbb)})$. From this, the claim about the uniqueness of the representation \eqref{eq:rep.G} easily follows.

To prove the representation \eqref{eq:rep.G}, it is sufficient to verify that each $\xi\in L\p2(\Om,\Gscr_T,\Pbb)$ can be represented as
\begin{equation}\label{eq:rep.xi}
\xi=\Ebb[\xi|\Gscr_0]+K\cdot X\p c_T+W\ast(\mu\p {(X,H)}-\nu\p {(X,H)})_T,\quad K\in\Lrm_T\p2(X\p c,\Gbb),\quad W\in\Gscr\p2_T(\mu\p {(X,H)},\Gbb),
\end{equation}
since $(\Hscr\p2_T(\Gbb),\|\cdot\|_2)$ is isomorphic to $(L\p2(\Om,\Gscr_T,\Pbb),\|\cdot\|_2)$. 
 
We show \eqref{eq:rep.xi} by an application of the monotone class theorem. We consider a time $u>T$ arbitrary but fixed. We recall that $\wt\Gscr_u=\Fscr_u\vee\Hscr_u$. As a first step, we prove that every $\wt\Gscr_u$-measurable and square integrable random variable $\xi$ has the representation
\begin{equation}\label{eq:rep.Gu}
\xi=\Ebb[\xi|\Gscr_0]+K\cdot X\p c_u+W\ast(\mu\p{(X,H)}-\nu\p{(X,H)})_u,\quad K\in\Lrm\p2(X\p c,\Gbb),\quad W\in\Gscr\p2(\mu\p X,\Gbb).
\end{equation}
Let us consider the system 
\[
\Cscr_u:=\{\xi\in\wt\Gscr_t:\ \xi=\zt(1-H_s),\quad \zt\ \textnormal{ bounded and } \Fscr_u\textnormal{-measurable},\ 0\leq s\leq u\}
\]
which, by Lemma \ref{lem:gen.sys}, generates $\wt\Gscr_u$. We now show \eqref{eq:rep.Gu} for $\xi\in\Cscr_u$. So, let $\xi=\zt(1-H_s)$, where $\zt$ is $\Fscr_u$-measurable and bounded and $0\leq s\leq u$ is arbitrary but fixed. Denoting by $A$ the Az\'ema supermartingale (cf.\ \eqref{eq:def.az.supm}), we define $\zt\p\prime:=\zt A_s$. Then $\zt\p\prime$ is $\Fscr_u$-measurable and bounded. Therefore, $Z\p\prime_t:=\Ebb[\zt\p\prime|\Gscr_t]$, $t\geq0$, is a bounded $\Gbb$-martingale {\pa with terminal value in $\wt\Gscr_u$. As a} consequence of the hypothesis $(\Hscr)$, $Z\p\prime_t=\Ebb[\zt\p\prime|\Fscr_t]$ (see \cite[Theorem 3.2(d)]{AJ17}), $t\geq0$, meaning that $Z\p\prime$ is in fact a bounded $\Fbb$-martingale. Hence, we find $K\in \Lrm\p2(X\p c,\Fbb)$ and $W\p\prime\in\Gscr\p2(\mu\p X,\Fbb)$ such that
\[
Z\p\prime_t=Z\p\prime_0+K\cdot X\p c_t+W\p\prime\ast(\mu\p X-\nu\p X)_t,\quad t\geq0.
\]
From Proposition \ref{prop:con.Z}, $A>0$ on $[0,\tau]$. Let $M$ be the martingale defined in \eqref{eq:def.M} and let $U:=\Escr(-M)$ be the stochastic exponential of $-M$. Then, from Theorem \ref{thm:G.com.AH} (iii), we get
\[
\zt(1-H_s)=\zt1_{[0,\tau)}(s)=\zt1_{[0,\tau)}(s)A_sA_s\p{-1}=Z\p\prime_uU\p s_u
\]
where $U\p s_t:=U_{s\wedge t}$, $t\geq0$. We define the bounded $\Gbb$-martingale $Z_t:=\Ebb[\zt(1-H_s)|\Gscr_t]$, $t\geq0$. Using the {\pa Dol\'eans-Dade} equation for $U$, i.e.,
\[
U_t=1-U_-\cdot M_t,\quad t\geq0,
\]
 and integration by parts, from \cite[Proposition II.1.30 b)]{JS00}, since $U$ and $Z\p\prime$ have no common jumps by hypothesis $(\Ascr)$, we get
\begin{equation}\label{eq:rep.X}
\begin{split}
Z_u&=Z\p\prime_uU\p s_u
\\&=Z\p\prime_0+U\p s_-\cdot Z\p\prime_u+Z\p\prime_-1_{[0,s]}\cdot U_u
\\&=Z\p\prime_0+U\p s_-K\cdot X\p c_u+U\p s_-W\p\prime\ast(\mu\p X-\nu\p X)_u-Z\p\prime_-U_-1_{[0,s]}\cdot M_u.
\end{split}
\end{equation}
We clearly have $Z\in\Hscr\p2(\Gbb)$, and hence $[Z,Z]\in\Ascr\p+(\Gbb)$. Furthermore, denoting by $Y\p j$, $j=1,2,3$, the first, the second and the third integral on the right-hand side in the previous formula, we see that, by hypothesis $(\Ascr)$, $[Y\p i, Y\p j]=0$, $i\neq j$. Therefore, we deduce the inclusions $1_{[0,u]}U\p s_-K\in \Lrm\p2(X\p c,\Gbb)$, $1_{[0,u]}U\p s_-W\p\prime\in\Gscr\p2(\mu\p X,\Gbb)$ and $Z\p\prime_-U_-1_{[0,s]}\in \Lrm\p2(M,\Gbb)$. We now define
\[
W(\om,t,x_1,x_2)=U\p s_{t-}(\om)W\p\prime(\om,t,x_1)1_{\{x_2=0\}}-Z\p\prime_-U_-1_{[0,s]}1_{\{x_1=0,x_2=1\}}.
\]
From Proposition \ref{prop:ju.mea.Y} and the continuity of $\Lm\p\Gbb$, we get
\[
\wt W_t(\om)=U_{t-}\p s(\om)\wt W\p\prime_t(\om)-Z\p\prime_{t-}(\om)U_{t-}(\om)1_{[0,s]}(t)\Delta H_t(\om),\quad t\geq0.
\]
Hence, we obtain the inclusion $1_{[0,u]}W\in\Gscr\p2(\mu\p {(X,H)},\Gbb)$ and the identity
\[
1_{[0,u]}W\ast(\mu\p {(X,H)}-\nu\p {(X,H)})=1_{[0,u]}U\p s_-W\p\prime\ast(\mu\p X-\nu\p X)-Z\p\prime_-1_{[0,u]}U_-1_{[0,s]}\cdot M,
\]
since both sides in the previous identity are purely discontinuous local martingale with the same jumps.
Therefore, \eqref{eq:rep.X} becomes
\[
\begin{split}
\zt(1-H_s)&=Z\p\prime_0+1_{[0,u]}U\p s_-K\cdot X\p c_u+1_{[0,u]}W\ast(\mu\p {(X,H)}-\nu\p {(X,H)})_u
\end{split},
\]
meaning that \eqref{eq:rep.Gu} holds for all $\xi\in\Cscr_u$ and the proof for this elementary case is complete. We denote by $\Bb(\wt\Gscr_u)$ the system of the $\wt\Gscr_u$-measurable and bounded random variables and let $\Kscr\subseteq \Bb(\wt\Gscr_u)$ be the subfamily of $\Bb(\wt\Gscr_u)$ consisting of the random variables which can be represented as in \eqref{eq:rep.Gu}. Then, from the previous step, $\Cscr_u\subseteq\Kscr$. Moreover, $\Kscr$ is a monotone class of $\Bb(\wt\Gscr_u)$: Take $(\xi\p n)_{n\geq1}\subseteq\Kscr$ uniformly bounded such that $\xi\p n\uparrow \xi$ as $n\rightarrow+\infty$. Then
\[
\begin{split}
\xi\p n&=\Ebb[\xi\p n|\Gscr_0]+K\p n\cdot X_u\p c+W\p n\ast(\mu\p{(X,H)}-\nu\p{(X,H)})_u\\&=\Ebb[\xi\p n|\Gscr_0]+1_{[0,u]}K\p n\cdot X_\infty\p c+1_{[0,u]}W\p n\ast(\mu\p{(X,H)}-\nu\p{(X,H)})_\infty
\end{split}
\]
with $K\p n\in \Lrm\p2(X\p c,\Gbb)$, $W\p n\in\Gscr\p2(\mu\p{(X,H)},\Gbb)$. Furthermore,
 $\xi$ is bounded, $\wt\Gscr_u\subseteq\Gscr_u$-measurable and, by dominated convergence, $\xi\p n\uparrow \xi$ in $L\p2(\Om,\wt\Gscr_u,\Pbb)$ as $n\rightarrow+\infty$. Thus, from Lemma \ref{lem:WRP.lim}, we have
\[
\xi=\Ebb[\xi|\Gscr_0]+K\cdot X\p c_\infty+W\ast(\mu\p{(X,H)}-\nu\p{(X,H)})_\infty,\quad K\in\Lrm\p2(X\p c,\Gbb),\quad W\in\Gscr\p2(\mu\p{(X,H)},\Gbb).
\]
Since the stochastic integrals on the right-hand side of the previous expression are $\Gbb$-martingales, by the $\wt\Gscr_u\subseteq\Gscr_u$-measurability of $\xi$, taking the conditional expectation with respect to $\Gscr_u$, we see that the inclusion $\xi\in\Kscr$ holds. The monotone class theorem (see \cite[Theorem 1.4]{HWY92}) now yields $\Bb(\wt\Gscr_u)\subseteq\Kscr$. Let $\xi\in L\p2(\Om,\wt\Gscr_u,\Pbb)$ and $\xi\geq0$. Defining $\xi\p n:=\xi\wedge n$, by dominated convergence, we see that $\xi\p n\longrightarrow\xi$ in $L\p2(\Om,\wt\Gscr_u,\Pbb)$ as $n\rightarrow+\infty$. Therefore, from the previous step and Lemma \ref{lem:WRP.lim}, 
\[
\xi=\Ebb[\xi|\Gscr_0]+K\cdot X\p c_\infty+W\ast(\mu\p{(X,H)}-\nu\p{(X,H)})_\infty,\quad K\in\Lrm\p2(X\p c,\Gbb),\quad W\in\Gscr\p2(\mu\p{(X,H)},\Gbb).
\]
Now, taking the conditional expectation with respect to $\Gscr_u$ as above, we see that $\xi$ can be represented as in \eqref{eq:rep.Gu}. For an arbitrary $\xi\in L\p2(\Om,\wt\Gscr_u,\Pbb)$, it is enough to apply the previous step to the positive and the negative part of $\xi$ and then to use the linearity of the involved stochastic integrals, to see that $\xi$ can be represented as in \eqref{eq:rep.Gu}. The proof of \eqref{eq:rep.Gu} for an arbitrary $\xi\in L\p2(\Om,\wt\Gscr_u,\Pbb)$ is now complete. Let now $\xi\in L\p2(\Om,\Gscr_T,\Pbb)$. Since $\xi$ is $\Gscr_T$-measurable and $u>T$, then $\xi$ is also $\wt\Gscr_u$-measurable. By the previous step, we see that $\xi$ can be represented as in \eqref{eq:rep.Gu}. Taking now the conditional expectation with respect to $\Gscr_T$ in \eqref{eq:rep.Gu} and using that the stochastic integrals on the right-and side of \eqref{eq:rep.Gu} are all $\Gbb$-martingales, by the $\Gscr_T$-measurability of $\xi$, we get \eqref{eq:rep.xi} for every $\xi\in L\p2(\Om,\Gscr_T,\Pbb)$. The proof of the theorem is now complete.
\end{proof}

Combining Theorem \ref{thm:wea.prp} and Proposition \ref{prop:eq.wprp}, we get the following corollary.

\begin{corollary}
Let $\tau$ be a random time satisfying Assumptions \ref{ass:av.im} and let $X$ be an $\Fbb$-semimartingale possessing the \textnormal{WRP} with respect to $\Fbb$. Let $T>0$ be an arbitrary but fixed finite time horizon. Then, the semimartingale $(X,H)\p{tr}$ possesses the $\Hscr\p p$-\textnormal{WRP} with respect to the filtration $\Gbb_T$, for every $p\geq1$, that is, every $N\in\Hscr\p p_T(\Gbb)$ can be represented as
\begin{equation}\label{eq:hpwrp.G}
N_t=N_0+K\cdot X\p c_t+W\ast(\mu\p{(X,H)}-\nu\p{(X,H)})_t,\quad t\in[0,T],\quad K\in\Lrm_T\p p(X\p c,\Gbb),\quad W\in\Gscr\p p_T(\mu\p{(X,H)},\Gbb).
\end{equation}
Furthermore, the representation \eqref{eq:hpwrp.G} is unique on $[0,T]$ in the following sense: If $K\p\prime\in\Lrm_T\p p(X\p c,\Gbb)$ and $W\p\prime\in\Gscr\p p_T(\mu\p{(X,H)},\Gbb)$ are other integrands such that \eqref{eq:hpwrp.G} holds, we then have $\|K-K\p\prime\|_{\Lrm_T\p p(X\p c,\Gbb)}=0$, $\|W-W\p\prime\|_{\Gscr\p p_T(\mu\p{(X,H)},\Gbb)}=0$ and $K\cdot X\p c_t=K\p\prime\cdot X\p c_t$, $W\ast(\mu\p{(X,H)}-\nu\p{(X,H)})_t=W\p\prime\ast(\mu\p{(X,H)}-\nu\p{(X,H)})_t$, for all $t\in[0,T]$ a.s.
\end{corollary}

\section{Applications in Exponential Utility Maximization}\label{sec:app.ex.ut.max}
In this section we consider a problem of exponential utility optimization of the {\pa expected} terminal wealth in presence of an additional exogenous risk source that cannot be inferred from the information available in the market, represented by the filtration $\Fbb$. The additional risk source can be a shock event, as the death of the investor or the default of part of the market. Its occurrence time {\pa modeled by} $\tau$.

The optimization problem described above will be solved following the \emph{dynamical approach}. This method is based on a martingale optimality principle obtained using suitable BSDEs. For the exponential utility function the related BSDE have a non-Lipschitz generator. However, the theory developed by Becherer in \cite{Be06}, in particular Theorem 3.5 therein, ensures the existence and the uniqueness of the solution of such BSDEs. The fundamental tool to apply results from \cite{Be06} to this context is Theorem \ref{thm:wea.prp}.

We now state some further assumptions which are mainly those of \cite{Be06}.
\begin{assumption}\label{ass:den.ass} Let $\tau:\Om\longrightarrow(0,+\infty]$ be a random time and {\pa let $T>0$} be  an arbitrary but fixed finite time horizon. We denote by $H$ the default process associated with $\tau$. A filtration $\Fbb$ satisfying the usual conditions is given and $\Gbb$ denotes the progressive enlargement of $\Fbb$ by $\tau$, while $\Gbb_T=(\Gscr_t)_{t\in[0,T]}$ is the restriction of $\Gbb$ to $[0,T]$. We also consider an $\Rbb\p d$-valued $\Fbb$-semimartingale $X$. The semimartingale characteristics of $X$ with respect to $\Fbb$ are $(B\p X,C\p X,\nu\p X)$. We furthermore require the following properties:

(1) $\tau$ satisfies Assumptions \ref{ass:av.im}.
\\[.2cm]
\indent(2) $\Fscr_0$ is trivial.
\\[.2cm]
\indent(3) The continuous local martingale part $X\p c$ of $X$ is a $\d$-dimensional standard Brownian motion $B=(B\p{1},\ldots,B\p{d})\p{tr}$. That is, $C\p X_t=t\mathrm{ Id}_d$, $\mathrm{Id}_d$ denoting the identity matrix in $\Rbb\p{d\times d}$.
\\[.2cm]
\indent(4) $X$ has the WRP with respect to $\Fbb$.
\\[.2cm] 
\indent(5) The $\Fbb$-predictable compensator $\nu\p X$ of the jump measure $\mu\p X$ is absolutely continuous, that is, there exists a L\'evy measure $\rho\p X$ on $\Rbb\p d$ (i.e., $\rho\p X$ is $\sig$-finite, $\int_{\Rbb\p d}(|x|\p2\wedge1)\rho\p X(\rmd x)<+\infty$ and $\rho\p X(\{0\})=0$) such that
\[
\nu\p X(\om,t,\rmd x)=\zt\p X(\om,t,x)\rho\p X(\rmd x)\rmd t,
\] 
where $(\om,t,x)\mapsto\zt(\om,t,x)\geq0$ is an $\Fbb$-predictable mapping. 
\\[.2cm]
\indent(6) The $\Gbb$-dual predictable projection $\Lm\p\Gbb$ of $H$ is absolutely continuous, that is, there exist a $\Gbb$-predictable process $\lm$ such that
\[
\Lm\p\Gbb_t=\int_0\p t1_{[0,\tau]}(s)\lm_s\rmd s.
\]
\\[.2cm]
\indent(7) The densities $\zt\p X(\om,t,x)$ and $\lm_t(\om)$ are uniformly bounded and the L\'evy measure $\rho\p X$ is finite.
\end{assumption} 
Assumptions \ref{ass:den.ass} (1) and (2) imply that the $\sig$-algebra $\Gscr_0$ is trivial, since $\Gscr_0=\Fscr_0\vee\sig(\{\tau=0\})$ (see \cite[Lemma 4.4 a)]{Jeu80}).

By Assumptions \ref{ass:den.ass} (1) and Lemma \ref{leb:XcharG}, $X$ is a semimartingale with characteristics $(B\p X,C\p X,\nu\p X)$ also with respect to $\Gbb$. Therefore, $(X,H)\p{tr}$ is a  $\Gbb$-semimartingale with values in $E=\Rbb\p d\times\{0,1\}$.

The semimartingale characteristics $(B\p {(X,H)},C\p {(X,H)},\nu\p {(X,H)})$ of $(X,H)\p{tr}$, from Assumptions \ref{ass:den.ass} and Proposition \ref{prop:ju.mea.Y} (ii), are given by $B\p {(X,H)}=(B\p X,\Lm\p\Gbb)\p{tr}$, 
\[C\p {(X,H)}_t=\begin{bmatrix}
    \mathrm{Id}_d       & 0 \\
    
    0       & 0
\end{bmatrix} t
\] 
and 
\[
\nu\p {(X,H)}(\om,\rmd t,\rmd x_1,\rmd x_2)=\zt\p {(X,H)}(\om,t,x_1,x_2)\rho\p {(X,H)}(\rmd x_1,\rmd x_2)\rmd t,
\]
where
\[
\zt\p {(X,H)}(\om,t,x_1,x_2):=\zt\p X(\om,t,x_1)1_{\{x_1\neq 0,x_2=0\}}+\lm_t(\om)1_{[0,\tau]}(\om,t)1_{\{x_1=0,x_2=1\}},
\]
and the measure $\rho\p {(X,H)}$ on $(E,\Bscr(E))$ is
\[
\rho\p {(X,H)}(\rmd x_1,\rmd x_2):=\rho\p X(\rmd x_1)\delta_0(\rmd x_2)+\delta_1(\rmd x_2)\delta_0(\rmd x_1).
\]
Clearly, $\zt\p {(X,H)}$ is uniformly bounded, as $\zt\p X$ and $\lm$ are, and $\rho\p {(X,H)}(E)<+\infty$ since $\rho\p X(\Rbb\p d)<+\infty$. 

From  (1) and (4) in  Assumptions \ref{ass:den.ass} and by Theorem \ref{thm:wea.prp}, the $\Rbb\p{d+1}$-dimensional semimartingale ${(X,H)\p{tr}}$ possesses the WRP with respect to $\Gbb_T$.

We remark that, because of \cite[Lemma 4.4 b)]{Jeu80} or \cite[Proposition 2.11 b)]{AJ17}, we can assume, without loss of generality, that the density $\lm$ of the process $\Lm\p\Gbb$ appearing in Assumptions \ref{ass:den.ass} (6) is, in fact, an $\Fbb$-predictable process.

\subsection{The Market Model}\label{subs:mar.mod}
We consider the same market model as in \cite[Section 4.1]{Be06} with respect to the filtration $\Fbb_T=(\Fscr_t)_{t\in[0,T]}$. The market price of risk $\varphi$ is an $\Fbb_T$-predictable and bounded $\Rbb\p d$-valued process. The volatility matrix $\sig$ is an  $\Rbb\p{d\times d}$-valued $\Fbb_T$-predictable process. We require that $\sig_t(\om)$ is invertible, for every $(\om,t)$ in $\Om\times[0,T]$ and, denoting by $\sig\p{i}$ the rows of $\sig$, we require that the euclidean norm of $\sig\p i_t$ is bounded. We assume that the price process $S=(S_t\p1,\ldots,S\p d_t)_{t\in[0,T]}$ evolves according to the following stochastic differential equation

\begin{equation}\label{eq:evol.S}
\rmd S_t=\mathrm{diag}(S_t\p i)_{i=1,\ldots,d}\sig_t(\varphi_t\rmd t+\rmd B_t),\quad S_0\in(0,+\infty)\p d,\quad t\in[0,T],
\end{equation}
where $\mathrm{diag}(S\p i)_{i=1,\ldots,d}$ takes values in $\Rbb\p{d\times d}$ and denotes the diagonal-matrix-valued process with $S$ on the diagonal. Denoting by $\Escr(Z)$ the stochastic exponential of the semimartingale $Z$, and setting
\[
\widehat B:=B+\int_0\p\cdot\varphi_s\rmd s,
\] 
from \eqref{eq:evol.S}, we deduce that $S\p i=S_0\p i\Escr(\sig\p i\cdot \widehat B)$. Notice that this market-model is free of arbitrage opportunities. Indeed,
\[
\rmd\Qbb:=\Escr\left(-\varphi\cdot B\right)_T\rmd\Pbb,
\]
defines a probability measure equivalent to $\Pbb$ on $\Gscr_T$. Since $\Gscr_T$ and $\Fscr_T$ contain the same null sets, namely those of $\Fscr$, $\Qbb$ is also equivalent to $\Pbb$ on $\Fscr_T$. By the boundedness of $\varphi$, Novikov's condition and Girsanov's theorem, $\widehat B$ is a $\Qbb$-Wiener process with respect to $\Gbb_T$ (and hence, with respect to $\Fbb_T$). Under $\Qbb$, again by Novikov's condition, $\sig\p i$ being bounded, $S\p i$ is a $\Qbb$-martingale with respect to $\Gbb_T$ (and hence with respect to $\Fbb_T$). Therefore, $\Qbb$ is an equivalent martingale measure for $S$ and the market model is free of arbitrage opportunities (with respect to both the filtrations $\Fbb_T$ and $\Gbb_T$). 

We remark that the $\Gbb$-predictable compensator $\nu\p X$ of $\mu\p X$ with respect to $\Qbb$ does not change (see \cite[Theorem 12.31]{HWY92}). 

\begin{definition}[Admissible Strategies]\label{def:adm.str}  An admissible strategy $\theta$ is a $\Gbb_T$-predictable $\Rbb\p d$-valued process satisfying the following conditions:

(i) $\int_0\p \cdot|\theta_s|\p2\rmd s\in L\p2(\Om,\Gscr_T,\Pbb)$.

(ii) $\exp(-\al\int_0\p\cdot\theta_s\rmd\widehat B_s)$ is a process of class $(D)$ with respect to the filtration $\Gbb_T$.

\noindent We denote by $\Theta$ the set of admissible strategies.
\end{definition}
We stress that the set $\Theta$ of the admissible strategies, consisting of $\Gbb_T$-predictable processes, can be regarded as the set of the strategies of an insider who has private information about the occurrence of $\tau$.

Let $\theta\in\Theta$ be an admissible strategy. The wealth process $X\p{x,\theta}=(X\p{x,\theta}_t)_{t\in[0,T]}$ is defined by
\begin{equation}\label{eq:wea.pr}
X\p{x,\theta}_t:=x+\int_0\p t\theta_s\rmd\widehat B_s=\int_0\p t\theta_s\big(\mathrm{diag}(S_t\p i)_{i=1,\ldots,d}\sig_t\big)\p{-1}\rmd S_s,\quad t\in[0,T].
\end{equation}
Clearly, for each $\theta\in\Theta$, the wealth process $X\p{x,\theta}$ is a $\Gbb_T$-martingale under the equivalent measure $\Qbb$ introduced above. Therefore, the set $\Theta$ of admissible strategies is free of arbitrage opportunities.

\subsection{Exponential Utility Maximization at $T>0$}\label{subs:ex.max.T}
Let $\xi$ belong to the class of bounded and $\Gscr_T$-measurable random variables {\pa(denoted by $\Bb(\Gscr_T)$)}. We now consider the optimization problem
\begin{equation}\label{eq:opt.pb}
U\p\xi(x)=\sup_{\theta\in\Theta}\Ebb\big[-\exp(-\al(X\p{x,\theta}_{T}-\xi))\big],\quad \xi\in\Bb(\Gscr_T),\quad \al>0.
\end{equation}
The random variable $\xi\in\Bb(\Gscr_T)$ represents a \emph{liability} or an \emph{asset} of the investor at maturity $T$.

We observe that, since the strategies are now $\Gbb_T$-predictable and $\xi\in\Bb(\Gscr_T)$, in general, this optimization problem cannot be solved in the filtration $\Fbb_T$, which represents the information available in the market, as done in \cite{Be06}. Furthermore, \eqref{eq:opt.pb} can be regarded as the optimization problem of an insider who has private information about the occurrence of $\tau$.
 
We stress that any $\Gscr_T$-measurable random variable $\xi$ can be regarded as a defaultable claim. As an example, we can consider claims of the form $\xi=\xi_11_{\{\tau>T\}}+\xi_21_{\{\tau\leq T\}}$, where $\xi_1$ is an $\Fscr_T$-measurable random variable, representing the pay-off of $\xi$ if the default does not occur in $[0,T]$ while $\xi_2$ is an $\Fscr_T$-measurable random variable representing the recovery pay-off of $\xi$ in case of default before $T$. For pricing method for defaultable claims, see, e.g., \cite{BJ08}, \cite{BJR04} and \cite{BJR08}.

To solve \eqref{eq:opt.pb}, we use the \emph{martingale optimality principle} on $[0,T]$, that is, we construct a family $\Rscr:=\{R\p\theta,\ \theta\in\Theta\}$ of $\Gbb_T$-adapted processes with the following properties:

(1) $R\p{\theta,x}_{T}=-\exp(-\al(X\p{\theta,x}_{T}-\xi))$, for every $\theta\in\Theta$.

(2) $R\p{\theta,x}_0\equiv r\p x$ is a constant not depending on $\theta$, for every $\theta\in\Theta$.

(3) $R\p{\theta,x}$ is a $\Gbb_T$-supermartingale for every $\theta\in\Theta$.

(4) There exists $\theta\p\ast\in\Theta$ such that $R\p{\theta\p\ast,x}$ is a $\Gbb_T$-martingale.

\noindent Notice that the strategy $\theta\p\ast$ in (4) above is optimal. Indeed, for any $\theta\in\Theta$ we get
\[
\Ebb\big[-\exp(-\al(X\p{\theta,x}_{T}-\xi))\big]=\Ebb\big[R\p{\theta,x}_{T}\big]\leq R\p{\theta,x}_{0}=r\p x=\Ebb\big[R\p{\theta\p\ast,x}_{T}\big]=\Ebb\big[-\exp(-\al(X\p{\theta\p\ast,x}_{T}-\xi))\big].
\]
We recall the notation $E:=\Rbb\p d\times\{0,1\}\subseteq\Rbb\p{d+1}$ and denote by $L\p0(\Bscr(E),\rho\p{(X,H)},\Rbb)$ the space of $\Bscr(E)$-measurable, $\Rbb$-valued functions on $E$ with the topology of the convergence in measure. We consider the generator 
\[
f:\Om\times[0,T]\times\Rbb\p d\times L\p0(\Bscr(E),\rho\p{(X,H)},\Rbb)\longrightarrow \Rbb
\] 
given by
\begin{equation}\label{def:gen}
\begin{split}
f(\om,t,&z,w_t):=-\bigg(z\p{tr}\varphi_t(\om)+\frac{|\varphi_t(\om)|\p2}{2\al}\bigg)\\&\quad+\frac{1}{\al}\,\int_E\big(\exp(\al w(t,x_1,x_2))-1-\al w(t,x_1,x_2)\big) \zt\p{(X,H)}(\om,t,x_1,x_2)\rho\p{(X,H)}(\rmd x_1,\rmd x_2).
\end{split}
\end{equation}
We then consider the BSDE
\begin{equation}\label{eq:bsde}
Y_t=\xi+\int_{t}\p{T} f(s,Z_s,W_s)\rmd s-\int_{t}\p{T} Z_s\rmd B_s-\int_{t}\p{T}\int_E  W(s,x_1,x_2)(\mu\p{(X,H)}-\nu\p{(X,H)})(\rmd s,\rmd x_1,\rmd x_2).
\end{equation}
Since $f$ satisfies the assumptions of \cite[Theorem 3.5]{Be06}, we deduce the existence and the uniqueness of a triplet $(Y,Z,W)\in\Sscr_T\p\infty(\Gbb)\times\Lrm_T\p2(B,\Gbb)\times\Gscr_T\p2(\mu\p{(X,H)},\Gbb)$ which satisfies \eqref{eq:bsde}, where $\Sscr_T\p\infty(\Gbb)$ is the space of essentially bounded $\Gbb_T$-semimartingales. Furthermore, $Y$ is bounded and $W$ is $\Pbb\otimes\nu\p{(X,H)}$-a.e.\ bounded. Without loss of generality, we assume that $W$ is bounded. Indeed, according to the proof of \cite[Theorem 3.5]{Be06}, there exists a bounded $\Gbb_T$-predictable mapping $W\p\prime$ such that $W=W\p\prime$ $\Pbb\otimes\nu\p{(X,H)}$-a.e.\ and, therefore, $W\ast(\mu\p{(X,H)}-\nu\p{(X,H)})$ is indistinguishable from $W\p\prime\ast(\mu\p{(X,H)}-\nu\p{(X,H)})$ on $[0,T]$. 

We define the family $\Rscr:=\{R\p\theta,\ \theta\in\Theta\}$ by
\begin{equation}\label{eq:def.fam.R}
R\p\theta_t:=-\exp(-\al(X\p\theta_t-Y_t)),\quad t\in[0,T].
\end{equation}
We now verify that $\Rscr$ fulfils the martingale optimality principle. We follow the ideas of \cite[Theorem 4.1]{Be06}. Before, we recall the following \emph{true-martingale} criterion for local martingales: Let $\Abb$ be an arbitrary right-continuous filtration. An $\Abb$-local martingale $X$ is a uniformly integrable martingale if and only if it is a process of class (D) with respect to $\Abb$ (see \cite[Proposition I.1.47 c)]{JS00}). For an example of a uniformly integrable local martingale which \emph{is not} a martingale see \cite[Chapter IV, Section 26 and Example VI.29]{DM82}.
\begin{theorem}\label{thm:R.mop}
The family $\Rscr$ satisfies the martingale optimality principle. Furthermore, the solution of the optimization problem \eqref{eq:opt.pb} is given by $\theta\p\ast=Z+\al\p{-1}\varphi\in\Theta$ and the explicit expression of the value function is $U\p\xi(x)=-\exp(-\al(x-Y_0))$.  
\end{theorem}
\begin{proof}
We preliminarily observe that, since $W$ and $\zt\p{(X,H)}$ are bounded and $\rho\p{(X,H)}$ is finite, both $W$ and $(\rme\p{\al W}-1)$ belong to $\Gscr\p2_T(\mu\p{(X,H)},\Gbb)$. Therefore, we deduce \[\al W\ast(\mu\p{(X,H)}-\nu\p{(X,H)})\in\Hscr_T\p2(\Gbb),\qquad(\rme\p{\al W}-1)\ast(\mu\p{(X,H)}-\nu\p{(X,H)})\in\Hscr\p2_T(\Gbb).\] Furthermore, we have $\Delta[(\rme\p{\al W}-1)\ast(\mu\p{(X,H)}-\nu\p{(X,H)})]\geq\rme\p{-\al c}-1>-1$, where $c>0$ is a constant such that $|W(\om,t,x_1,x_2)|\leq c$. 

Applying It\^o's formula to $\exp\big(\al W\ast(\mu\p{(X,H)}-\nu\p{(X,H)})-(\rme\p{\al W}-1-\al W)\ast\nu\p{(X,H)}\big)$ and then using {\pa the Dol\'eans--Dade} equation for the stochastic exponential, we verify the identity
\begin{equation}\label{eq:st.log.ex}
\exp\big(\al W\ast(\mu\p{(X,H)}-\nu\p{(X,H)})-(\rme\p{\al W}-1-\al W)\ast\nu\p{(X,H)}\big)=\Escr\big((\rme\p{\al W}-1)\ast(\mu\p{(X,H)}-\nu\p{(X,H)})\big).
\end{equation}
Using now \eqref{eq:st.log.ex}, BSDE \eqref{eq:bsde}, the explicit form \eqref{def:gen} of the generator $f$ and the properties of the stochastic exponential, we verify that, for every $\theta\in\Theta$, the identity $R\p\theta=\rme\p{-\al(x-Y_0)}A\p\theta\Escr(H\p\theta)$ holds, where
\begin{gather*}
A\p\theta:=-\exp\Bigg(\frac{\al\p2}2\int_0\p \cdot|\theta_s-Z_s-\al\p{-1}\varphi_s|\p2\rmd s\Bigg),\\H\p\theta:=-\al(\theta-Z)\cdot B+(\exp(\al W)-1)\ast(\mu\p{(X,H)}-\nu\p{(X,H)}).
\end{gather*}
Since  $(\theta-Z)\in\Lrm_T\p2(B,\Gbb)$ and $\Delta H\p\theta>-1+\delta$, for some $\delta>0$, we deduce $H\p\theta\in\Hscr\p2_T(\Gbb)$, $\Escr(H\p\theta)>0$, and finally
\begin{equation}\label{eq:pb.H}
\begin{split}
\aPP{H\p\theta}{H\p\theta}=\al\p2\int_0\p\cdot |\theta_s-&Z_s|\p 2\rmd s\\&+\int_0\p\cdot\int_E(\exp(\al W(s,x_1,x_2))-1)\p2\zt\p {(X,H)}(s,x_1,x_2)\rho\p {(X,H)}(\rmd x_1,\rmd x_2)\rmd s.
\end{split}
\end{equation}
Since $Y$ is bounded and $\theta\in\Theta$,  by definition (cf.\ \eqref{eq:def.fam.R}) $R\p\theta$ is a process of class (D) with respect to $\Gbb_T$. Hence, from $\Escr(H\p\theta)_t\leq-\rme\p{\al(x-Y_0)}R\p\theta_t$, we get that $\Escr(H\p\theta)$ is a $\Gbb_T$-local martingale of class (D), and hence, a $\Gbb_T$-uniformly integrable martingale, for every $\theta\in\Theta$. The process $A\p\theta$ is continuous, hence locally bounded, and decreasing. If now $(\sig_n)_n$ is a localizing sequence of $\Gbb_T$-stopping times for $A\p\theta$, i.e., $A\p\theta_{\cdot\wedge\sig_n}$ is bounded, and $0\leq s\leq t\leq T$, we get
\[
\begin{split}
\Ebb[R\p\theta_{t\wedge\sig_n}|\Gscr_s]&\leq \rme\p{-\al(x-Y_0)}A\p\theta_{\sig_n\wedge s}\Ebb[\Escr(H\p\theta)_{t\wedge\sig_n}|\Gscr_s]\\&=\rme\p{-\al(x-Y_0)}A\p\theta_{\sig_n\wedge s}\Escr(H\p\theta)_{s\wedge\sig_n}\longrightarrow R\p\theta_s,\quad n\rightarrow+\infty.
\end{split}
\]

But $(R\p\theta_{t\wedge\sig_n})_n$ is uniformly integrable for each $\theta\in\Theta$. Thus, $\Ebb[R\p\theta_{t\wedge\sig_n}|\Gscr_s]\longrightarrow \Ebb[R\p\theta_{t}|\Gscr_s]$ as $n\rightarrow+\infty$. So, $R\p\theta$ is a $\Gbb_T$-supermartingale, for every $\theta\in\Theta$. We consider the process $\theta\p\ast:=Z+\al\p{-1}\varphi$, which is $\Gbb_T$-predictable. We are going to verify that $\theta\p\ast\in\Theta$ and that $R\p{\theta\p\ast}$ is a martingale. For this end, we show that $\Escr(H\p{\theta\p\ast})$ is a uniformly integrable $\Gbb_T$-martingale. From \eqref{eq:pb.H}, $\varphi$ and $W$ being bounded, by the boundedness of $\zt\p{(X,H)}$ and the finiteness of $\rho\p{(X,H)}$, we deduce that $\aPP{H\p{\theta\p\ast}}{H\p{\theta\p\ast}}_T$ is bounded. Since $H\p{\theta\p\ast}$ is also a martingale with bounded jumps, \cite[Theorem 10.9]{HWY92} implies that $H\p{\theta\p\ast}$ is a $\Gbb_T$-martingale belonging to the class $\textnormal{BMO}_T(\Gbb)$, that is, the class of the $\Gbb$-martingale in the class BMO on the finite time interval $[0,T]$. Furthermore, $\Delta H\p{\theta\p\ast}>-1+\delta$, for some $\delta>0$. Hence, \cite[Theorem 2]{IS79} implies that $\Escr(H\p{\theta\p\ast})$ is a uniformly integrable $\Gbb_T$-martingale and therefore a process of class (D). Since $R\p{\theta\p\ast}=-\rme\p{-\al(x-Y_0)}\Escr(H\p{\theta\p\ast})$, we immediately get that $R\p{\theta\p\ast}$ is a $\Gbb_T$-martingale which is moreover uniformly integrable and, hence, of class (D). Therefore, by the boundedness of $Y$, we deduce that $\exp(-\al\int_0\p\cdot\theta\p\ast_s\rmd\widehat B_s)$ is a process of class (D). Clearly, $\theta\p\ast\in\Lrm_T\p2(B,\Gbb)$ holds, and, therefore, $\theta\p\ast\in\Theta$. By the martingale optimality principle, we deduce that $\theta\p\ast$ solves \eqref{eq:opt.pb} and $U\p\xi(x)=-\rme\p{-\al(x-Y_0)}$. The proof of the proposition is now complete.
\end{proof}

Let us now assume that $\xi\in\Bb(\Fscr_T)$, that is, $\xi$ is a bounded and $\Fscr_T$-measurable random variable. In this special case, we compare BSDE \eqref{eq:bsde} with BSDE (4.13) in \cite{Be06}. From Assumptions \ref{ass:den.ass}, we see that
\begin{align}
f(\om,t,z,w_t)=&-\bigg(z\p{tr}\varphi_t(\om)+\frac{|\varphi_t(\om)|\p2}{2\al}\bigg)\nonumber\\&+\nonumber\frac{1}{\al}\,\int_{\Rbb\p d}\big(\exp(\al w(t,x_1,0))-1-\al w(t,x_1,0)\big) \zt\p{ X}(\om,t,x_1)\rho\p{X}(\rmd x_1)\\&+\frac{1}{\al}\,\big(\exp(\al w(t,0,1))-1-\al w(t,0,1)\big)\lm_t(\om)1_{[0,\tau]}(\om,t)\label{eq:com.ter3}
\end{align}
and so $(\om,t,w_t)\mapsto f(\om,t,z,w_t)$ is, in general, a $\Gbb_T$-predictable function. 
Therefore, to obtain the generator of BSDE (4.13) in \cite{Be06}, which we denote by $g$, it is enough to take $w(t,0,1)\equiv 0$ in \eqref{eq:com.ter3}, that is, for $w\in L\p0(\Bscr(E),\rho\p{(X,H)},\Rbb)$, defining $u_t(x)=w_t(x,0)$, we obtain $g(\om,t,z,u_t)=f(\om,t,z,w_t)\big|_{w_t(0,1)=0}$. Hence, BSDE (4.13) in \cite{Be06} correspond to BSDE \eqref{eq:bsde} with $\xi\in\Bb(\Fscr_T)$ and generator $g$. The existence and the uniqueness of an $\Fbb_T$-solution follows from \cite[Theorem 3.5]{Be06} using the WRP of $X$ with respect to $\Fbb$.

\begin{remark}[Indifference pricing of defaultable claims]\label{subs:ind.pri}
We stress that an application of this section can be given if one considers the problem of indifference pricing of defaultable claims. Indeed, it is evident that the market model which we consider is not complete neither in the filtrations $\Fbb_T$ nor in $\Gbb_T$, since the price process $S$ is continuous but both $\Fbb_T$ and $\Gbb_T$ support martingales with jumps. Hence, the problem of pricing a $\Gscr_T$-measurable claim $\xi$ arises. A well-known way to price a contingent claim $\xi\in\Bb(\Gscr_T)$ is \emph{indifference pricing}: The \emph{indifference price} or \emph{utility indifference value} $\pi$ of the contingent claim $\xi$ is given by the implicit solution of the equation
\[
U\p0(x)=U\p\xi(x+\pi).
\]
Hence, the utility indifference value $\pi$ is the value that, if added to the initial capital $x$, makes the investor indifferent (in terms of expected utility) between only trading or trading and selling $\xi$ for $\pi$ in $t=0$ then trading and paying $\xi$ in $T$.

By Theorem \ref{thm:R.mop} (see also \cite[Theorem 4.4]{Be06}) the utility indifference value is $\pi=Y_0\p\xi-Y_0\p 0$,
where $Y\p0$ is the solution of BSDE \eqref{eq:bsde} {\pa in $\Fbb_T$ with $\xi=0$ and generator $g$ (defined above, just before this remark) while $Y\p\xi$ is the is the solution of BSDE \eqref{eq:bsde} in $\Gbb_T$ with terminal condition $\xi\in\Bb(\Gscr_T)$ and generator $f$}. We stress that it is important to ensure that the defaultable claim $\xi$ is a $\Gscr_T$-measurable and \emph{bounded} random variable. For the problem of indifference pricing of defaultable claims in the case of a progressively enlarged Brownian filtration $\Fbb$ we refer, e.g., to \cite{BJ08}, \cite{BJR04} or \cite{LQ11}. The problem of the indifference price has been also considered in \cite{Be06}.  
\end{remark}

\subsection{Exponential Utility Maximization at $T\wedge\tau$}\label{subs:ex.max.rth}
We now consider the optimization problem
\begin{equation}\label{eq:opt.pb.rh}
\hat U\p\xi(x)=\sup_{\hat\theta\in\hat\Theta}\Ebb\big[-\exp(-\al(X\p{\hat\theta}_{T\wedge\tau}-\xi))\big],\quad \xi\in\Bb(\Gscr_{T\wedge\tau}),\quad \al>0
\end{equation}
where $\Bb(\Gscr_{T\wedge\tau})$ denotes the class of bounded and $\Gscr_{T\wedge\tau}$-measurable random variables. The optimization problem \eqref{eq:opt.pb.rh} describes the case in which the investor can only follow his investment up to the occurrence time $\tau$ of the exogenous shock event. In other words, the investor has access to the market only up to time $\tau$. This means that the price process for the investor is not $S=(S_t)_{t\in[0,T]}$ itself but rather $S\p \tau=(S_{t}\p\tau)_{t\in[0,T]}$, where $S\p\tau_t:=S_{t\wedge\tau}$, $t\in[0,T]$. Notice that the optimization problem \eqref{eq:opt.pb.rh} automatically arises from \eqref{eq:opt.pb} if the price process $S$ in \eqref{eq:opt.pb.rh} is substituted by $S\p\tau$ and $\xi$ is assumed $\Gscr_{T\wedge\tau}$-measurable.

Jeanblanc et al.\ studied in \cite{JMPR15} the problem \eqref{eq:opt.pb.rh} when $\Fbb_T$ is the filtration generated by a $d$-dimensional Brownian motion $B$. In \cite{JMPR15} admissible strategies take values in a closed subset $C\subseteq\Rbb\p d$, which represents an additional constraint-set for admissible strategies.
The set of constraints $C$ leads to BSDEs with a non-Lipschitz generator $f$ which does not fit in the frame of \cite{Be06} because of a quadratic term in $z$. Furthermore, in \cite{JMPR15} the authors do not require that the density $\lm$ in Assumptions \ref{ass:den.ass} is bounded, as we do. On the other side, in \cite{JMPR15} the authors assume that the conditional law of $\tau$ given $\Fscr_t$ is equivalent to the Lebesgue measure and the immersion property (see \cite[\textbf{(H1)} (Density hypothesis)]{JMPR15}): These assumptions on $\tau$ seem to be stronger than Assumptions \ref{ass:den.ass} (1) and (6).

In summary, if we consider the special case $C=\Rbb\p d$ and $\lm$ bounded, then \eqref{eq:opt.pb.rh} can be seen as a generalization of the situation considered in \cite{JMPR15} in a Brownian setting to the case of an underlying filtration $\Fbb_T$ supporting martingales with jumps.
 
The optimization problem \eqref{eq:opt.pb.rh} has been considered  in \cite{JMPR15} as a separate problem. Thanks to the WRP with respect to $\Gbb_T$, which we obtained in Theorem \ref{thm:wea.prp} for all $\Gbb_T$-martingales, we can now deduce the existence and the uniqueness of the solution of BSDE \eqref{eq:bsde} on the random time interval $[0,T\wedge\tau]$ from \cite[Theorem 3.5]{Be06}. This allows to solve the optimization problem \eqref{eq:opt.pb.rh} exactly in the same way as \eqref{eq:opt.pb}. For this aim, we need the following proposition, which is a corollary to \cite[Theorem 3.5]{Be06}. Therefore, in the proof of this result, we use the notation from \cite[Theorem 3.5]{Be06}.

For a $\Gbb_T$-stopping time $\sig$, $Y\p\sig$  denotes the process $Y$ stopped at $\sig$, i.e., $Y\p\sig_t:=Y_{t\wedge\sig}$, $t\in[0,T]$. 

\begin{proposition}\label{prop:cor.to.Bec}
Let $\sig$ be a $\Gbb$-stopping time, let $\xi\in\Bb(\Gscr_{T\wedge\sig})$ and let $f$ be the generator defined in \eqref{def:gen}. Then the \emph{BSDE}
\begin{equation}\label{eq:stopped}
\begin{split}
Y_{t\wedge\sig}=\xi+\int_{t\wedge\sig}\p{T\wedge\sig} &f(s,Z_s,W_s)\rmd s-\int_{t\wedge\sig}\p{T\wedge\sig} Z_s\rmd B_s\\&-\int_{t\wedge\sig}\p{T\wedge\sig}\int_E  W(s,x_1,x_2)(\mu\p{(X,H)}-\nu\p{(X,H)})(\rmd s,\rmd x_1,\rmd x_2),\quad t\in[0,T],
\end{split}
\end{equation}
admits a unique bounded solution which is given by the solution \[(Y\p\sig,1_{[0,\sig]}Z,1_{[0,\sig]}W)\in\Sscr\p\infty_T(\Gbb)\times\Lrm_T\p2(B,\Gbb)\times\Gscr_T\p2(\mu\p{(X,H)},\Gbb),\quad W\ \textnormal{ bounded},\] of \emph{BSDE} \eqref{eq:bsde} with generator $f_\sig:=1_{[0,\sig]}f$ and terminal condition $\xi$.
\end{proposition} 
\begin{proof}
Since $f$ as in \eqref{def:gen} satisfies the assumptions of \cite[Theorem 3.5]{Be06}, also $f_\sig=1_{[0,\sig]}f$ does. Therefore, there exists the unique solution $(Y,Z,W)\in\Sscr_T\p\infty(\Gbb)\times\Lrm\p2_T(B,\Gbb)\times\Gscr\p2_T(\mu\p{(X,H)},\Gbb)$ of \eqref{eq:bsde} with generator $f_\sig$. Let $\wt f$ and $\wt f_\sig$ be the truncated generators defined as in the proof of \cite[Theorem 3.5]{Be06}. Then $\wt f_\sig$ satisfies the assumption of \cite[Proposition 3.2]{Be06} and therefore $(Y,Z,W)$ is also solution of BSDE \eqref{eq:bsde} with generator $\wt f_\sig$ (this is shown in the proof of \cite[Theorem 3.5]{Be06}). Hence, we have
\begin{equation}\label{eq:rep.Y}
Y_t=\Ebb\left[\xi+\int_t \p T \wt f_\sig(s,Z_s,W_s)\,\rmd s\Big|\Gscr_t\right]
\end{equation}
and
\begin{equation}\label{eq:rep.rv}
\xi+\int_0\p {T}\wt f_\sig(s,Z_s,W_s)\,\rmd s=\Ebb\left[\xi+\int_0\p {T}\wt f_\sig(s,Z_s,W_s)\,\rmd s\right]+Z\cdot B_T+W\ast(\mu\p{(X,H)}-\nu\p{(X,H)})_T.
\end{equation}
Since $\xi\in\Bb(\Gscr_{T\wedge\sig})$ and $\wt f_\sig=1_{[0,\sig]}\wt f$, from \eqref{eq:rep.rv} and Doob's stopping theorem, we get
\[
\xi+\int_0\p {T\wedge\sig}\wt f(s,Z_s,W_s)\,\rmd s=\Ebb\left[\xi+\int_0\p {T\wedge\sig}\wt f(s,Z_s,W_s)\,\rmd s\right]+1_{[0,\sig]}Z\cdot B_T+1_{[0,\sig]}W\ast(\mu\p{(X,H)}-\nu\p{(X,H)})_T.
\]
From \eqref{eq:rep.Y} we deduce $Y_{T\wedge\sig}=\xi$ and $Y\p\sig=Y$. So $(Y\p\sig,1_{[0,\sig]}Z,1_{[0,\sig]}W)$ solves BSDE \eqref{eq:bsde} with generator $\wt f_\sig$ and terminal condition $\xi\in\Bb(\Gscr_{T\wedge\sig})$. But then, according to the proof of \cite[Theorem 3.5]{Be06}, we also have that $(Y\p\sig,1_{[0,\sig]}Z,1_{[0,\sig]}W)$ solves BSDE \eqref{eq:bsde} with generator $f_\sig$ and terminal condition $\xi\in\Bb(\Gscr_{T\wedge\sig})$. Therefore, $(Y\p\sig,1_{[0,\sig]}Z,1_{[0,\sig]}W)$ also satisfies BSDE \eqref{eq:stopped} and the proof of the proposition is complete.
\end{proof}
We remark that, because of Proposition \ref{prop:cor.to.Bec}, we are able to consider the optimization problem \eqref{eq:opt.pb.rh} on the interval $[0,T\wedge\sig]$, for every $\Gbb$-stopping time $\sig$.
{\pa We denote by $(Y,1_{[0,\sig]}Z,1_{[0,\sig]})W$ solution of \eqref{eq:stopped} with $Y\in\Sscr_T\p\infty(\Gbb)$ and $W$ bounded.}
\\[1em]
\indent We now discuss the solution of the optimization problem \eqref{eq:opt.pb.rh}. First we define the set $\hat\Theta$ of the admissible strategies. Let $\Theta$ be the set of admissible strategies for the optimization problem \eqref{eq:opt.pb} introduced in Definition \ref{def:adm.str}. For any $\theta\in\Theta$, we define $\hat\theta:=1_{[0,T\wedge\tau]}\theta$ and clearly $\hat\theta\in\Theta$ holds. Let now $X\p{\theta,x}$ be the wealth process defined in \eqref{eq:wea.pr}. Then $X\p{\hat\theta,x}=X\p{\theta,x}$ on $[0,T\wedge\tau]$ and $(X\p{\theta,x})\p{T\wedge\tau}=X\p{\hat\theta,x}$ holds. So, the set of admissible strategies for \eqref{eq:opt.pb.rh} can be restricted to $\hat\Theta:=\{\theta\in\Theta:1_{(T\wedge\tau,T]}\theta=0\}\subseteq\Theta$. 
\begin{remark}[Admissible strategies for the random-time horizon problem]\label{rem:rh.str}
We now make the following important consideration: Let $A$ be a $\Gbb$-predictable process. Then, because of \cite[Lemma 4.4. b)]{Jeu80}, there exists an $\Fbb$-predictable process $a$ such that $A1_{[0,\tau]}=a1_{[0,\tau]}$, showing that $\Fbb$-predictable and $\Gbb$-predictable processes coincide on $[0,\tau]$. This means that the set $\hat\Theta$ of admissible strategies for the optimization problem \eqref{eq:opt.pb.rh}, consists of strategies which are actually $\Fbb_T$-predictable. In other words, if the  optimization problem \eqref{eq:opt.pb} on the time horizon $[0,T]$ can be regarded as the problem of an insider who can use $\Gbb_T$-predictable strategies, having some private information about $\tau$, the optimization problem \eqref{eq:opt.pb.rh} actually describes the problem of an agent for whom the available information is exclusively the one in the market (that is, he pursues $\Fbb_T$-predictable strategies) but, for some reasons, he has only access to the market up to the occurrence of the exogenous shock event, whose occurrence time is modelled by $\tau$. 
\end{remark}
Let now $\xi\in\Bb(\Gscr_{T\wedge\tau})$. From Proposition \ref{prop:cor.to.Bec} we know that BSDE \eqref{eq:stopped} on $[0,T\wedge\tau]$ with generator $f$ and terminal condition $\xi$, corresponds to BSDE \eqref{eq:bsde} on $[0,T]$ with generator $1_{[0,T\wedge\tau]}f$ and terminal condition $\xi$. Hence, BSDE \eqref{eq:stopped} has the unique solution 
\[
{\pa (Y,1_{[0,T\wedge\tau]}Z,1_{[0,T\wedge\tau]}W)\in\Sscr\p\infty_T(\Gbb)\times\Lrm_T\p2(B,\Gbb)\times\Gscr_T\p2(\mu\p{(X,H)},\Gbb),}
\] 
where $W$ is furthermore bounded. Notice that, being a $\Gbb_T$-predictable process, $Z$  coincides with an $\Fbb_T$-predictable process on $[0,\tau]$ (see \cite[Proposition 2.11]{AJ17}). A similar statement holds also for $W$: That is $W$ coincides with a $\Pscr(\Fbb)\otimes\Bscr(E)$-measurable mapping on $[0,\tau]$. To see this, it is enough to consider a bounded $\Gbb$-predictable mapping $G$ of the form $G(\om,t,x_1,x_2)=g_t(\om)f(x_1,x_2)$, where $g$ is a bounded $\Gbb$-predictable process and $f$ a bounded measurable function. For $\Gbb$-predictable mappings of this form the statement clearly hold, because of  \cite[Proposition 2.11]{AJ17}. Furthermore, this is a system generating $\Pscr(\Gbb)\otimes\Bscr(E)$. By the monotone class theorem, we get the result for every bounded $\Pscr(\Gbb)\otimes\Bscr(E)$-measurable mapping $G$ and, by approximation, for every nonnegative and then for every $\Pscr(\Gbb)\otimes\Bscr(E)$-measurable mapping $G$.

We now define the family $\hat\Rscr=\{\hat R\p{\hat\theta},\ \hat\theta\in\hat\Theta\}$ by 
\[
\hat R\p{\hat\theta}_{t\wedge\tau}=-\exp(-\al(X\p{\hat\theta,x}_{t\wedge\tau}-Y_{t\wedge\tau})),\quad t\in[0,T],\quad \hat\theta\in\hat\Theta.
\]
 Then, $\hat R\p{\hat \theta}_{T\wedge\tau}=-\exp(-\al(X\p{\hat\theta}_{T\wedge\tau}-\xi))$ holds. If we now define, for $t\in[0,T]$,
{\pa\[
\hat A\p{\hat\theta}_{t\wedge\tau}:=-\exp\Bigg(\frac{\al\p2}2\int_0\p{t\wedge\tau}|\hat\theta_s-Z_s-\al\p{-1}
\varphi_s|\p2\rmd s\Bigg)
\]
and
\[\begin{split}
\hat H_{t\wedge\tau}\p{\hat\theta}:=-\al\int_0\p{t\wedge\tau}\big((\hat\theta_s-&Z_s)\big)\rmd B_s\\&+\int_0\p{t\wedge\tau}\int_E(\exp(\al W(s,x_1,x_2))-1)\big)\ast(\mu\p{(X,H)}-\nu\p{(X,H)})(\rmd s,\rmd x_1,\rmd x_2),
\end{split}
\]
}
we can verify as in the proof of Theorem \ref{thm:R.mop} the identity $\hat R\p{\hat \theta}=\rme\p{-\al(x-Y_0)}\hat A\p{\hat\theta}\Escr(\hat H\p{\hat \theta})$, for $\hat\theta\in\hat\Theta$. Furthermore, $\hat R\p{\hat\theta}_{\cdot\wedge\tau}$ is a supermartingale on $[0,T]$ for every $\hat\theta\in\hat\Theta$. If now $\theta\p\ast$ is the optimal strategy for the optimization problem \eqref{eq:opt.pb} given in Theorem \ref{thm:R.mop}, it follows that $\hat\theta\p\ast:=1_{[0,T\wedge\tau]}\theta\p\ast$ belongs to $\hat\Theta$ and $\hat\R\p{\hat\theta\p\ast}_{\cdot\wedge\tau}$ is a martingale. Therefore, $\hat\Rscr$ satisfies the martingale optimality principle on $[0,T\wedge\tau]$ and $\theta\p\ast=1_{[0,T\wedge\tau]}(Z+\al\p{-1}\varphi)$ is the optimal solution of \eqref{eq:opt.pb.rh}. In particular, the explicit {\pa expression} of the value function is given by $\hat U\p\xi(x)=\Ebb[-\exp(X\p{\hat\theta\p\ast,x}_{T\wedge\tau}-\xi)]=-\exp(-\al(x-Y\p{T\wedge\tau}_0))$.

\paragraph{The continuous case.} We stress that, if $\rho\p X=0$, that is, $X$ is a $d$-dimensional Brownian motion, and $\Fbb=\Fbb\p X$, then \eqref{eq:stopped} with $\sig=\tau$ becomes
\begin{equation}\label{eq:rh.con}
Y_{t\wedge\tau}=\xi+\int_{t\wedge\tau}\p{T\wedge\tau} f(s,Z_s,U_s)\rmd s-\int_{t\wedge\tau}\p{T\wedge\tau} Z_s\rmd B_s-\int_{t\wedge\tau}\p{T\wedge\tau} U_s\rmd M_s,\quad t\in[0,T]
\end{equation}
where, for $t\in[0,T\wedge\tau]$, the generator is given by and
\[
f(\om,t,z,u)=-\bigg(z\p{tr}\varphi_t(\om)+\frac{|\varphi_t(\om)|\p2}{2\al}\bigg)+\frac{1}{\al}\,\big(\exp(\al u)-1-\al u\big)\lm_t(\om)1_{[0,\tau]}(\om,t),
\]
The previous equation has the unique solution {\pa $(Y,1_{[0,T\wedge\tau]}Z,1_{[0,T\wedge\tau]}U)$}, which is the solution of the same equation on $[0,T]$ with generator $1_{[0,T\wedge\tau]}f$ and initial condition $\xi\in\Bb(\Gscr_{T\wedge\tau})$. We notice that we can rewrite
\[
Y_{t\wedge\tau}=\xi-\int_{t\wedge\tau}\p{T\wedge\tau} g(s,Z_s,U_s)\rmd s-\int_{t\wedge\tau}\p{T\wedge\tau} Z_s\rmd B_s-\int_{t\wedge\tau}\p{T\wedge\tau} U_s\rmd H_s,\quad t\in[0,T]
\]
where $g(\om,t,z,u):=-f(\om,t,z,u)-\lm_t u$, which is the BSDE considered in \cite{JMPR15}. Hence, BSDE \eqref{eq:rh.con} corresponds to \cite[Eq.\ (3.5)]{JMPR15} with $C=\Rbb\p d$ and a bounded intensity $\lm$. This means that, in this special case, we can recover \cite[Theorem 4.17] {JMPR15}. Furthermore, because of the uniqueness of the bounded solution, the solution of \eqref{eq:rh.con} and the solution given in \cite[Proposition 4.4]{JMPR15} coincide.
\paragraph*{Acknowledgement.}
I acknowledge Martin Keller-Ressel (TU Dresden) and funding from the German Research Foundation (DFG) under grant ZUK 64.

I gratefully  thank {\pa Hans-J\"urgen Engelbert} (FSU Jena) for useful discussions and comments.


\begin{thebibliography}{28}
\providecommand{\natexlab}[1]{#1}
\providecommand{\url}[1]{\texttt{#1}}
\expandafter\ifx\csname urlstyle\endcsname\relax
  \providecommand{\doi}[1]{doi: #1}\else
  \providecommand{\doi}{doi: \begingroup \urlstyle{rm}\Url}\fi

\bibitem[Aksamit and Jeanblanc(2017)]{AJ17}
A.~Aksamit and M.~Jeanblanc.
\newblock \emph{Enlargement of filtration with finance in view}.
\newblock Springer, 2017.

\bibitem[Aksamit et~al.(May 2018)Aksamit, Jeanblanc, and Rutkowski]{AJR18}
A.~Aksamit, M.~Jeanblanc, and M.~Rutkowski.
\newblock Predictable representation property for progressive enlargements of a
  poisson filtration.
\newblock \emph{Stoch. Proc. Appl.}, Available online, May 2018.

\bibitem[Becherer(2006)]{Be06}
D.~Becherer.
\newblock Bounded solutions to backward sdes with jumps for utility
  optimization and indifference hedging.
\newblock \emph{Ann. Appl. Probab.}, 16\penalty0 (4):\penalty0 2027--2054,
  2006.

\bibitem[Bielecki and Jeanblanc(2008)]{BJ08}
T.~Bielecki and M.~Jeanblanc.
\newblock Indifference pricing of defaultable claims.
\newblock In R.~Carmona, editor, \emph{Indifference pricing: theory and
  applications}, chapter~6, pages 211--240. Princeton University Press, 2008.

\bibitem[Bielecki et~al.(2004)Bielecki, Jeanblanc, and Rutkowski]{BJR04}
T.~Bielecki, M.~Jeanblanc, and M.~Rutkowski.
\newblock Hedging of defaultable claims.
\newblock In J.-M. et~al. Morel, editor, \emph{Paris-Princeton Lectures on
  Mathematical Finance 2003}, chapter~1, pages 1--132. Springer, 2004.

\bibitem[Bielecki et~al.(2008)Bielecki, Jeanblanc, and Rutkowski]{BJR08}
T.~Bielecki, M.~Jeanblanc, and M.~Rutkowski.
\newblock Pricing and trading credit default swaps in a hazard process model.
\newblock \emph{Ann. Appl. Probab.}, 18\penalty0 (6):\penalty0 2495--2529,
  2008.

\bibitem[Blanchet-Scalliet et~al.(2008)Blanchet-Scalliet, El~Karoui, M., and
  Martellini]{BSEKJM08}
C.~Blanchet-Scalliet, N.~El~Karoui, Jeanblanc M., and L.~Martellini.
\newblock Optimal investment decisions when time-horizon is uncertain.
\newblock \emph{Journal of Mathematical Economics}, 44\penalty0 (11):\penalty0
  1100--1113, 2008.

\bibitem[Bouchard and Pham(2004)]{BoPh04}
B.~Bouchard and H.~Pham.
\newblock Wealth-path dependent utility maximization in incomplete markets.
\newblock \emph{Finance Stoch.}, 8\penalty0 (4):\penalty0 579--603, 2004.

\bibitem[Choulli and Yansori(2018)]{CY18}
T.~Choulli and S.~Yansori.
\newblock Deflators and log-optimal portfolios under random horizon: Explicit
  description and optimization.
\newblock \emph{arXiv preprint arXiv:1803.10128}, 2018.

\bibitem[Choulli et~al.(2018)Choulli, Daveloose, and Vanmaele]{CDV18}
T.~Choulli, C.~Daveloose, and M.~Vanmaele.
\newblock A martingale representation theorem and valuation of defaultable
  securities.
\newblock \emph{arXiv preprint arXiv:1805.11844}, 2018.

\bibitem[Coculescu et~al.(2012)Coculescu, Jeanblanc, and Nikeghbali]{CJN12}
D.~Coculescu, M.~Jeanblanc, and A.~Nikeghbali.
\newblock Default times, no-arbitrage conditions and changes of probability
  measures.
\newblock \emph{Finance Stoch.}, 16\penalty0 (3):\penalty0 513--535, 2012.

\bibitem[Dellacherie(1972)]{D72}
C.~Dellacherie.
\newblock \emph{Capacit\'es et processus stocahstiques}.
\newblock Springer, Berlin, 1972.

\bibitem[Dellacherie and Meyer(1982)]{DM82}
C.~Dellacherie and P.-A. Meyer.
\newblock \emph{Probabilities and {Potential}{ A} and {B}}.
\newblock North-Holland Publishing Company, 1982.

\bibitem[Di~Tella and Engelbert(2016)]{DTE16}
P.~Di~Tella and H.-J. Engelbert.
\newblock The chaotic representation property of compensated-covariation stable
  families of martingales.
\newblock \emph{Ann. Probab.}, 44\penalty0 (6):\penalty0 3965--4005, 2016.

\bibitem[{\'E}mery(1989)]{E89}
M.~{\'E}mery.
\newblock On the {Az\'ema} martingales.
\newblock \emph{S{\'e}minaire de Probabilit{\'e}s XXIII}, 1372:\penalty0
  66--87, 1989.

\bibitem[He et~al.(1992)He, J., and Yan]{HWY92}
S.~He, Wang J., and J.~Yan.
\newblock \emph{Semimartingale theory and stochastic calculus}.
\newblock Taylor \& Francis, 1992.

\bibitem[Izumisawa et~al.(1979)Izumisawa, Sekiguchi, and Shiota]{IS79}
M.~Izumisawa, T.~Sekiguchi, and Y.~Shiota.
\newblock Remark on a characterization of bmo-martingales.
\newblock \emph{Tohoku Mathematical Journal, Second Series}, 31\penalty0
  (3):\penalty0 281--284, 1979.

\bibitem[Jacod(1975)]{J75}
J.~Jacod.
\newblock Multivariate point processes: predictable projection, radon-nikodym
  derivatives, representation of martingales.
\newblock \emph{Z. Wahrscheinlichkeit.}, 31\penalty0 (3):\penalty0 235--253,
  1975.

\bibitem[Jacod(1979)]{J79}
J.~Jacod.
\newblock \emph{Calcul stochastique et probl{\`e}mes de martingales.}
\newblock Springer, 1979.

\bibitem[Jacod and Shiryaev(2003)]{JS00}
J.~Jacod and A.~Shiryaev.
\newblock \emph{Limit theorems for stochastic processes}.
\newblock Springer-Verlag, Berlin, second edition, 2003.

\bibitem[Jeanblanc and Song(2015)]{JS15}
M.~Jeanblanc and S.~Song.
\newblock Martingale representation property in progressively enlarged
  filtrations.
\newblock \emph{Stochastic Process. Appl.}, 125\penalty0 (11):\penalty0
  4242--4271, 2015.

\bibitem[Jeanblanc et~al.(2015)Jeanblanc, Mastrolia, D., and
  R{\'e}veillac]{JMPR15}
M.~Jeanblanc, T.~Mastrolia, Possama{\"\i} D., and A.~R{\'e}veillac.
\newblock Utility maximization with random horizon: {A BSDE} approach.
\newblock \emph{Int. J. Theor. Appl. Fin.}, 18\penalty0 (07):\penalty0 1550045,
  2015.

\bibitem[Jeulin(1980)]{Jeu80}
T.~Jeulin.
\newblock \emph{Semi-martingales et grossissement d'une filtration}, volume 833
  of \emph{Lecture notes in Mathematics}.
\newblock Springer, 1980.

\bibitem[Kunita(2004)]{K04}
H.~Kunita.
\newblock Representation of martingales with jumps and applications to
  mathematical finance.
\newblock \emph{Stochastic analysis and related topics in Kyoto}, 41, 2004.

\bibitem[Kunita and Watanabe(1967)]{KW67}
H.~Kunita and S.~Watanabe.
\newblock On square integrable martingales.
\newblock \emph{Nagoya Math. J.}, 30:\penalty0 209--245, 1967.

\bibitem[Kusuoka(1999)]{Ku99}
S.~Kusuoka.
\newblock A remark on default risk models.
\newblock In \emph{Adv. Math. Econ.}, volume~1, pages 69--82. Springer, 1999.

\bibitem[Lim and Quenez(2011)]{LQ11}
T.~Lim and M.-C. Quenez.
\newblock Exponential utility maximization in an incomplete market with
  defaults.
\newblock \emph{Electron. J. Probab.}, 16:\penalty0 1434--1464, 2011.

\bibitem[Wu and Gang(1982)]{WG82}
H.-S. Wu and W.-J. Gang.
\newblock The property of predictable representation of the sum of independent
  semimartingales.
\newblock \emph{Z. Wahrscheinlichkeit.}, 61\penalty0 (1):\penalty0 141--152,
  1982.

\end{thebibliography}
\end{document}